\documentclass[12pt,reqno]{amsart}
\usepackage[dvipsnames]{xcolor}
\usepackage{graphicx}
\usepackage{tikz-cd}
\usepackage{amsmath, amssymb}
\usepackage{mathrsfs}
\usepackage{amsthm}
\usepackage{multicol}
\usepackage{color}
\usepackage{bm}
\usepackage{mathtools}
\usepackage{nameref}
\usepackage{url}
\usepackage[colorlinks=true,linkcolor=blue!60!black,citecolor=teal!80!black,urlcolor=RoyalBlue]{hyperref}
\usepackage[utf8x]{inputenc} 
\usepackage[left=2.7cm,right=2.7cm,top=2.7cm,bottom=2.7cm]{geometry}
\usepackage{enumitem}
\usepackage{accents}
\usepackage{import}

\usepackage[color=red!20!,bordercolor=red, shadow]{todonotes}
\newtheorem{thm}{Theorem}[section]

\newtheorem{prop}[thm]{Proposition}

\newtheorem{dfn&prop}[thm]{Definition and Proposition}
\newtheorem{dfn&thm}[thm]{Definition and Theorem}

\theoremstyle{definition}
\newtheorem{defn}[thm]{Definition}
\newtheorem{observation}[thm]{Observation}

\newtheorem*{structure}{Structure of the article}

\newtheorem{example}[thm]{Example}
\newtheorem{discussion}[thm]{}

\theoremstyle{remark}
\newtheorem*{Claim}{Claim}

\newtheorem*{remark}{Remark}
\newtheorem*{notations}{Notation}

\numberwithin{equation}{section}

\newenvironment{subproof}{\begin{proof}[Proof of claim.]}{%
	\end{proof}}

\newcommand{\Deriv}{{\rm D}}
\def\phi{\varphi}

\def\B{{\mathcal{B}}}

\def\C{{\mathbb{C}}}
\def\D{{\mathbb{D}}}
\def\N{{\mathbb{N}}}

\def\R{{\mathbb{R}}}

\newcommand{\unbdd}[1]{\accentset{\infty}{#1}}

\newcommand{\Ima}{\operatorname{Im}}
\newcommand{\Rea}{\operatorname{Re}}

\newcommand{\Crit}{\operatorname{Crit}}

\newcommand{\Addr}{\operatorname{Addr}}
\newcommand{\ul}{\underline}

\newcommand{\diam}{\text{diam}}
\newcommand{\ultau}{\underline{\tau}}
\newcommand{\T}{\mathcal{T}}

\newcommand{\AV}{\operatorname{AV}}
\newcommand{\CV}{\operatorname{CV}}
\newcommand{\Orb}{\operatorname{Orb}}

\def\s{{\underline s}}

\newcommand*{\defeq}{\mathrel{\vcenter{\baselineskip0.5ex \lineskiplimit0pt
			\hbox{\scriptsize.}\hbox{\scriptsize.}}}%
	=}

\newcommand{\eqdef}{=\mathrel{\vcenter{\baselineskip0.5ex \lineskiplimit0pt
			\hbox{\scriptsize.}\hbox{\scriptsize.}}}}

\title{Combinatorics of criniferous entire maps with escaping critical values}

\author[L. Pardo-Sim\'{o}n]{Leticia Pardo-Sim\'{o}n}
\address{Institute of Mathematics of the Polish Academy of Sciences \\ ul.  \'Sniadeckich 8 \\ 00-656 Warsaw \\ Poland \\ 
	 \textsc{\newline \indent 
	   \href{https://orcid.org/0000-0003-4039-5556%
	     }{\includegraphics[width=1em,height=1em]{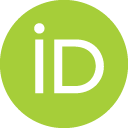} {\normalfont https://orcid.org/0000-0003-4039-5556}}
	       }}
\email{l.pardo-simon@impan.pl}
\subjclass[2010]{Primary 37F10; secondary 30D05.}

\begin{document}

\begin{abstract}
A transcendental entire function is called \textit{criniferous} if every point in its escaping set can eventually be connected to infinity by a curve of escaping points. Many transcendental entire functions with bounded singular set have this property, and this class has recently attracted much attention in complex dynamics. In the presence of escaping critical values, these curves \textit{break} or \textit{split} at (preimages of) critical points. In this paper, we develop combinatorial tools that allow us to provide a complete description of the escaping set of any criniferous function without asymptotic values on its Julia set. In particular, our description precisely reflects the splitting phenomenon. This combinatorial structure provides the foundation for further study of this class of functions. For example, we use these results in \cite{mio_splitting} to give the first full description of the topological dynamics of a class of transcendental entire maps with unbounded postsingular set.
\end{abstract}

\vspace{-10pt}
\maketitle

\section{Introduction}

Let $p$ be a polynomial of degree $d \geq 2$. Then, by Böttcher's Theorem, there is a conjugacy between $p$ and $z \mapsto z^{d}$ in a neighbourhood of infinity. Whenever all the orbits of the critical points of $p$ are bounded, or, equivalently, when its Julia set $J(p)$ is connected, this conjugacy can be extended to a biholomorphic map between $\C \setminus \overline{\D}$ and the basin of infinity of $p$, that we denote by $I(p)$. In particular, \textit{dynamic rays} of $p$ are defined as the curves that arise as preimages of radial rays from $\partial \D$ to $\infty$ under this conjugacy, and provide a natural foliation of  $I(p)$. If $J(p)$ is additionally locally connected, then each ray has a unique accumulation point in $J(p)$, and we say that it \textit{lands}. Dynamic rays and their landing behaviour are a powerful combinatorial tool in the study of polynomial dynamics, see \cite{Orsaynotes}.

For a transcendental entire map, $f$, infinity is an essential singularity, and consequently, Böttcher's Theorem no longer applies. In this paper, we restrict ourselves to the widely studied \textit{Eremenko-Lyubich class} $\B$, consisting of all transcendental entire maps, $f$, with bounded \textit{singular set} $S(f)$, which is the closure of the set of its critical and asymptotic values. Then, it is known that for some $f\in \B$, every point in their \textit{escaping set} \[I(f)\defeq \{z\in\C : f^n(z)\to \infty \text{ as } n\to \infty\}\]
can be connected to infinity by an escaping curve, called \textit{dynamic ray} by analogy with the polynomial case, e.g. \cite{Baranski_Trees, RRRS, mio_newCB}. More precisely, we adopt \cite[Definition 2.2]{RRRS} and \cite[Definition~1.2]{lasse_dreadlocks}:
\begin{defn}[Dynamic rays, criniferous maps]\label{def_ray}
Let $f$ be a transcendental entire function. A \emph{ray tail} of $f$ is an injective curve $\gamma :[t_0,\infty)\rightarrow I(f)$, with $t_0>0$, such that
\begin{itemize}
\item for each $n\geq 1$, $t \mapsto f^{n}(\gamma(t))$ is injective with $\lim_{t \rightarrow \infty} f^{n}(\gamma(t))=\infty$;
\item $f^{n}(\gamma(t))\rightarrow \infty$ uniformly 
in $t$ as $n\rightarrow \infty$.
\end{itemize}
A \emph{dynamic ray} of $f$ is a maximal injective curve $\gamma :(0,\infty)\rightarrow I(f)$ such that the restriction $\gamma_{|[t,\infty)}$ is a ray tail for all $t > 0$. We say that $\gamma$ \emph{lands} at $z$ if $\lim_{t \rightarrow 0^+} \gamma(t)=z$, and we call $z$ the \emph{endpoint} of $\gamma$. Moreover, we say that $f$ is \emph{criniferous} if for every $z\in I(f)$, there is $N\defeq N(z)\in \N$ so that $f^n(z)$ is in a ray tail for all $n\geq N$.
\end{defn}
Not all functions in $\B$ are criniferous, see e.g., \cite{RRRS, lasse_arclike, lasse_dreadlocks}. Nonetheless, it is possible to define symbolic dynamics for the points of any $f\in \B$ whose orbits stay away from a neighbourhood of $S(f)$. To this end, roughly speaking, one can partition a neighbourhood of infinity into topological half-strips, and call a sequence $\s\defeq F_0F_1F_2\ldots$ of such domains an \textit{(external) address}. Then, a point $z\in J(f)$ is said to have address $\s$ if $f^n(z)\in F_n$ for all $n\in \N$, and one may study the sets of points sharing their address, see \S \ref{sec_symbolic}. In particular, we show in Theorem \ref{prop_criniferous} that for a criniferous function $f\in\B$, these sets contain a collection of dynamic rays to which all points in $I(f)$ are eventually mapped. External addresses, or variants thereof, have been extensively used with great success in the study of functions in~$\B$, e.g., \cite{lasse_schleicher_combinatorics, Schleicher_Zimmer_exp,lasseSiegel,  lasse_dreadlocks,VassoBcn_innerfunctions,RRRS, lasse_arclike, lasse_Devaneyfast,BarKarp_codingtrees, Benini_Fagella_separation}. 

We note that for any $f\in \B$, all points in $I(f)$ whose orbit stays sufficiently far from the origin have a well-defined unique external address. However, it is desirable to have combinatorics for all points in $I(f)$. Recently, Benini and Rempe, \cite{lasse_dreadlocks}, developed combinatorics for escaping points of any $f\in \B$ with bounded \emph{postsingular set} $P(f) \defeq\overline{\bigcup_{n\geq 0}f^n (S(f))}$. In a rough sense, their approach consists on \textit{extending} the sets of points sharing an address to some sort of maximal connected sets, named \textit{dreadlocks}, that are not necessarily curves. In particular, this combinatorial structure allowed them to prove an analogue of Douady-Hubbard landing theorem for periodic dynamic rays for all $f\in \B$ with bounded $P(f)$.

\begin{figure}[htb]
	\centering
\begingroup%
  \makeatletter%
  \providecommand\color[2][]{%
    \errmessage{(Inkscape) Color is used for the text in Inkscape, but the package 'color.sty' is not loaded}%
    \renewcommand\color[2][]{}%
  }%
  \providecommand\transparent[1]{%
    \errmessage{(Inkscape) Transparency is used (non-zero) for the text in Inkscape, but the package 'transparent.sty' is not loaded}%
    \renewcommand\transparent[1]{}%
  }%
  \providecommand\rotatebox[2]{#2}%
  \newcommand*\fsize{\dimexpr\f@size pt\relax}%
  \newcommand*\lineheight[1]{\fontsize{\fsize}{#1\fsize}\selectfont}%
  \ifx\svgwidth\undefined%
    \setlength{\unitlength}{396.8503937bp}%
    \ifx\svgscale\undefined%
      \relax%
    \else%
      \setlength{\unitlength}{\unitlength * \real{\svgscale}}%
    \fi%
  \else%
    \setlength{\unitlength}{\svgwidth}%
  \fi%
  \global\let\svgwidth\undefined%
  \global\let\svgscale\undefined%
  \makeatother%
	\resizebox{0.56\textwidth}{!}{
  \begin{picture}(1,0.71428571)%
    \lineheight{1}%
    \setlength\tabcolsep{0pt}%
    \put(0,0){\includegraphics[width=\unitlength,page=1]{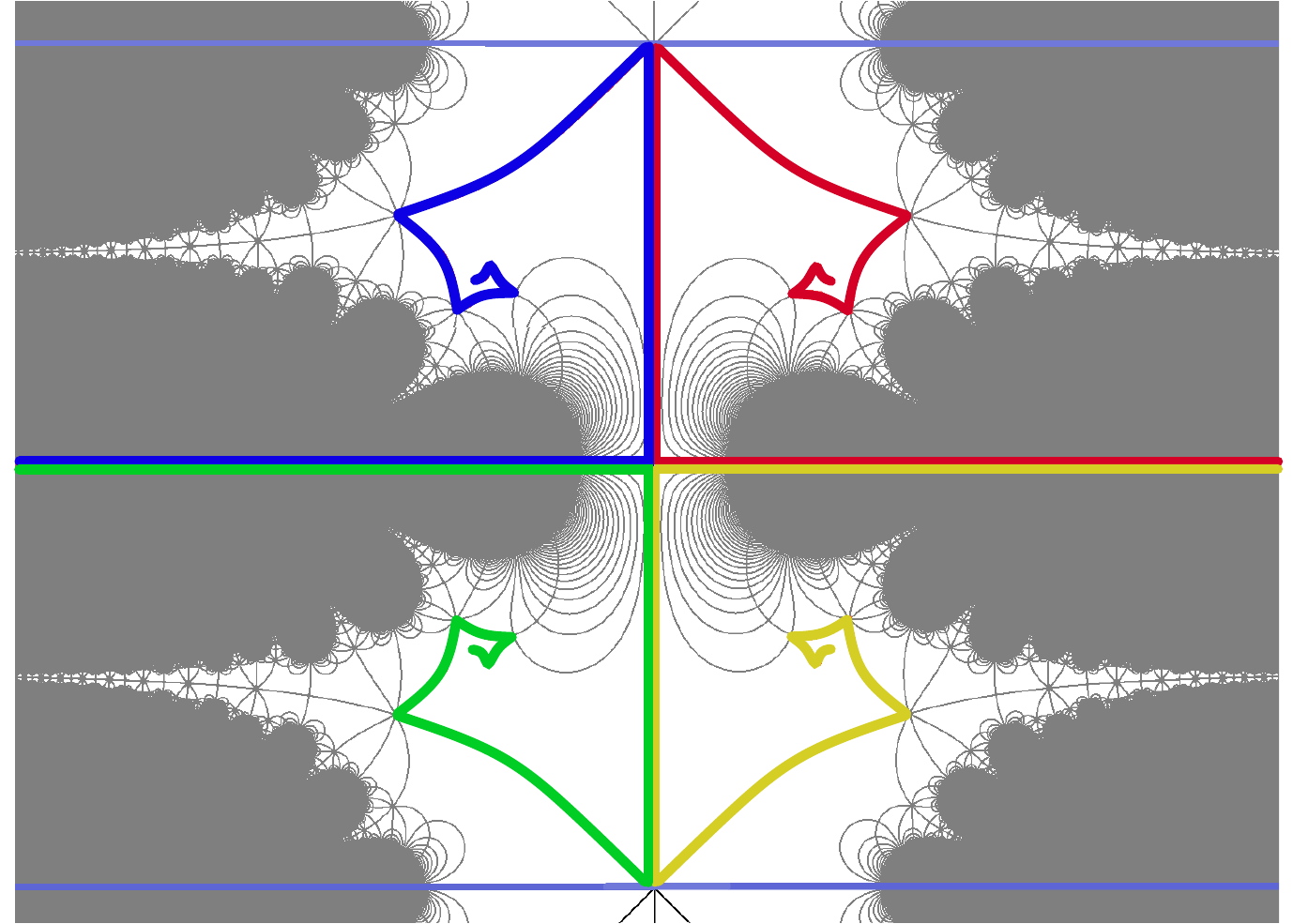}}%
  \end{picture}%
}
\endgroup
	\caption[Caption for LOF]{In colour, dynamic rays of $f(z)=\cosh^2(z)$ around the critical point zero. Further iterated preimages of the real axis are depicted in~grey.\protect\footnotemark}
	\label{fig:rays_cosh222}
\end{figure}
\footnotetext{Original picture by L. Rempe, modified for this paper. It first appeared in \cite[p.~166]{mio_thesis}.}

In this paper, we develop combinatorics for certain criniferous functions with unbounded postsingular set. We start by highlighting some challenges that escaping critical values present in the transcendental case. First, notice that for polynomials, the orbit of any critical value is either bounded, or converges to the super-attracting fixed point at infinity. Furthermore, in the latter case, it is still possible to define dynamic rays as the orthogonal trajectories of level curves for the Green's function in a natural way; see \cite[Appendix A]{Goldberg_Milnor} or \cite[\S 2.2]{kiwi_rationalrays}. However, the Julia set of any transcendental entire map contains a large set of points whose orbits are neither bounded nor escaping \cite{dave_osb_bungee}, and with the essential singularity at infinity, we find very different dynamics. In particular, concerning dynamic rays.

To illustrate the phenomena that we encounter, let us consider the map $f(z)\defeq \cosh^2(z).$ The dynamics of $f$ have already been explored in \cite{ripponFast}, where it is shown that $I(f)$ (and, in fact, its \textit{fast escaping set}) is connected. The singular set of $f$ consists of the critical values $0$ and $1$, with $f(0)=1$, and $1$ escaping to infinity along the positive real axis. Note that $0,i\pi/2$ and $-i\pi/2$ are critical points, and it is easy to check that $(-\infty, 0]$ and $[0, \infty)$ are both ray tails. The vertical segments $[0, -i\pi/2]$ and $[0, i\pi/2]$ are mapped univalently to $[0,1]\subset \R^+$, and thus, the union of each segment with either one of the ray tails $(-\infty, 0]$ and $[0, \infty)$, forms a different ray tail. We can think of this structure as four ray tails that partially overlap pairwise, see Figure \ref{fig:rays_cosh222}. Then, their endpoints $-i\pi/2$ and $i\pi/2$ are preimages of $0$ and critical points, and so the structure described repeats twice at each point. Hence, it can be understood as eight ray tails that overlap pairwise. By looking at further preimages of zero, this process can be continued. Our approach suggests considering two copies of each of the tails of $f$ near infinity, and trying to extend each copy in a careful and systematic way.

To fully describe the phenomenon of ``splitting'' or overlapping of rays at critical points, we introduce the concept of signed addresses for every criniferous $f\in \B$ with escaping critical values. More precisely, let $\Addr(f)$ be the set of external addresses of $f$. Then, we consider the set of \textit{signed addresses} $\Addr(f)_\pm\defeq \Addr(f)\times \{-,+\}$, that we endow with a topology such that each $z\in I(f)$ has at least two signed addresses that depend continuously on $z$. The following theorem summarizes the main results on this paper.

\begin{thm} \label{thm_signed_intro} Let $f\in \B$ be criniferous such that $J(f)\cap \AV(f)=\emptyset$. Then, its escaping set is a collection of dynamic rays
\begin{equation}
\{\Gamma(\s,\ast)\}_{(\s,\ast)\in \Addr(f)_\pm},
\end{equation}
that overlap piecewise between (preimages of) critical points. For each $(\s,\ast)\in \Addr(f)_\pm$, we can write 
\begin{equation*}
\Gamma(\s,\ast)=\bigcup_{n\geq 0}\gamma^n_{(\s, \ast)},
\end{equation*}
so that for every $n\in \N$, $\gamma^n_{(\s, \ast)}$ is a ray tail, $\gamma^n_{(\s, \ast)}\subset \gamma^{n-1}_{(\s, \ast)}$, and there exists a neighbourhood $\tau_n(\s, \ast)\supset \gamma^n_{(\s,\ast)}$, on which there is a well-defined inverse branch 
\begin{equation}\label{eq_inversebranch}
f^{-1,[n]}_{(\s,\ast)}\defeq \left(f\vert_{ \tau_n(\s, \ast)}\right)^{-1} \colon f(\tau_n(\s, \ast)) \longrightarrow \tau_n(\s, \ast),
\end{equation}  
that coincides for all signed addresses in an open interval of $(\s,\ast)$.
\end{thm}

The tools developed in this paper allow us to provide in the sequel \cite{mio_splitting}, a total description of the Julia set of certain criniferous functions with escaping critical values, introduced in \cite{mio_orbifolds} and \cite{mio_newCB}; see also \cite{mio_cosine} for similar results in the cosine family.

\begin{structure}
We start recalling in Section 2 the definition of \textit{(external) addresses} in terms of \textit{fundamental domains} for functions in $\B$. We provide the set of addresses with a topology and study, for criniferous functions, the sets of points sharing the same external address. In \S \ref{sec_signed} we extend this concept for criniferous functions: we define \textit{signed addresses} and show that for in the absence of asymptotic values, all points in their escaping set have at least two signed addresses. We associate to each signed address a nested union of escaping curves that we call \textit{canonical tails}. Next, in \S\ref{sec_hands}, we introduce the concept of \textit{fundamental hands} as preimages of certain subsets of fundamental domains on which inverse branches are well-defined, and so that for each canonical tail, we can find an inverse branch that contains the tail on its image. In particular, the concept of fundamental hands extends that of \textit{fundamental tails} from \cite{lasse_dreadlocks}. Theorem \ref{thm_signed_intro} will follow from the combination of these results.
\end{structure}

\subsection*{Basic notation}
As introduced throughout this section, Julia and escaping set of an entire function $f$ are denoted by $J(f)$ and $I(f)$ respectively. The set of critical values is $\CV(f)$, that of asymptotic values is $\AV(f)$, and the set of critical points will be $\Crit(f)$. The set of singular values of $f$ is $S(f)$, and $P(f)$ denotes the postsingular set. We denote the complex plane by $\C$ and a disk of radius $R$ by $\D_R$. We will indicate the closure of a domain $U$ by $\overline{U}$, which must be understood to be taken in $\C$. $A\Subset B$ means that $A$ is compactly contained in $B$. For a holomorphic function $f$ and a set $A$, $\Orb^{-}(A)\defeq \bigcup^{\infty}_{n=0} f^{-n}(A)$ and $\Orb^{+}(A)\defeq \bigcup^{\infty}_{n=0} f^{n}(A)$ are the respective backward and forward orbits of $A$ under $f$.

\subsection*{Acknowledgements}
I am very grateful to my supervisors Lasse Rempe and Dave Sixsmith for their continuous help and support. I also thank Daniel Meyer and Phil Rippon for valuable comments.
\section{External addresses and symbolic dynamics}\label{sec_symbolic}

The concept of \textit{external address} for functions in $\B$ allows to assign symbolic dynamics to points whose orbit stays away from a neighbourhood of their singular set. In this section, we review this notion, and study properties of the sets of points sharing a same external address, with special emphasis on criniferous functions. 
\begin{defn}[Tracts, fundamental domains]\label{def_fund}
Fix $f\in\B$ and let $D$ be a bounded Jordan domain around the origin, containing $S(f)$ and $f(0)$. Each connected component of $f^{-1}(\C\setminus \overline{D})$ is a \emph{tract} of $f$, and $\mathcal{T}_f$ denotes the set of all tracts. Let $\delta$ be an arc connecting a point of $\overline{D}$ to infinity in the complement of the closure of the tracts. Denote
\begin{equation}\label{eq_deffund}
\mathcal{W} \defeq \C\setminus( \overline{D}\cup\delta).
\end{equation}
Each connected component of $f^{-1}(\mathcal{W})$ is a \emph{fundamental domain} of $f$, and we call the collection of all of them an \textit{alphabet of fundamental domains}, that we denote $\mathcal{A}(D,\delta)$. Moreover, for each $F\in \mathcal{A}(D,\delta)$, $\unbdd{F}$ is the unbounded connected component of $F\setminus \overline{D}$. 
\end{defn}

For basic properties of tracts and fundamental domains, see for example \cite[Proposition 2.19]{mio_thesis} or \cite[\S3]{lasse_dreadlocks}. A function $f\in \B$ is of said to be of \textit{disjoint type} if $P(f)\Subset \C\setminus J(f)$ and $\C\setminus J(f)$ is connected. Alternatively, disjoint type maps can be characterized as those for which there is a choice of tracts whose boundaries are disjoint from the boundary of their image. More precisely: 
\begin{prop}[Characterization of disjoint type maps]\label{prop_charac_disjoint} A function $f\in \B$ is of disjoint type if and only if there exists a Jordan domain $D \supset S(f)$ such that $\overline{f(D)} \subset \overline{D}$.
\end{prop} 	
See for example \cite[Proposition 2.8]{helenaSemi} for a proof of Proposition \ref{prop_charac_disjoint}. The partition of $f^{-1}(\mathcal{W})$ into fundamental domains allows us to assign symbolic dynamics to those points whose orbit stays in $\mathcal{W}$.
\begin{defn}[External addresses for functions in $\B$] \label{def_extaddr}
Let $f\in\B$ and let $\mathcal{A}(D,\delta)$ be an alphabet of fundamental domains. An \emph{(infinite) external address} is a sequence $\s= F_0 F_1 F_2 \dots$ of elements in $\mathcal{A}(D,\delta)$. For each external address $\s$, we let
\begin{equation}\label{eq_Js}
J_{\s}\defeq\left\{z\in\C\colon f^n(z)\in\unbdd{F_n} \text{ for all $n\geq 0$}\right\}.
\end{equation}
We say that $\s$ is \emph{admissible} if $J_{\s}$ is non-empty, and we denote by $\Addr(f)$ the set of all admissible external addresses. If $z\in J_\s$ for some $\s \in \Addr(f)$, then we say that $z$ \textit{has (external) address} $\s.$ Moreover, $\sigma$ stands for the one-sided \emph{shift operator} on external addresses. That is, $\sigma(F_0 F_1 F_2\ldots ) = F_1F_2 \ldots$. In particular,
\begin{equation}\label{eq_inclJs}
f(J_\s)\subseteq J_{\sigma({\s})} \quad \text{ for all } \quad\s \in\Addr(f).
\end{equation}
\end{defn}

\begin{notations} Let $\s=s_0s_1s_2\ldots\in \Addr(f)$ and suppose that there exists $N\geq 0$ such that $s_i=s_N$ for all $i\geq N$. Then we write $\s=s_0\ldots s_{N-1}\overline{s_N}$.
\end{notations}
\begin{remark}
For the reader familiar with \cite{lasse_dreadlocks}, we note that the sets ``$J_{\s}$'' defined in \eqref{eq_Js}, are denoted by ``$J^0_\s(f)$'' in \cite[Definition 2.4]{lasse_dreadlocks}, and do not equal the sets ``$J_\s$'' introduced in \cite[Definition 4.2]{lasse_dreadlocks}. We have waived consistency in notation across articles in favour of simplifying notation in ours. Moreover, we note that these sets lie in $J(f)$, see \cite[Lemma 2.6]{lasse_dreadlocks}. Thus, we sometimes refer to them as \textit{Julia constituents}.
\end{remark}

\begin{observation}[Points with external address]\label{obs_addrz} Whenever it is defined, the external address of a point is unique because Julia constituents are by definition pairwise disjoint. If $f$ is of disjoint type, then by Proposition \ref{prop_charac_disjoint}, there is an alphabet $\mathcal{A}$ of fundamental domains so that $\unbdd{F}=F$ for all $F\in \mathcal{A}$. This implies that $J(f)$ is the disjoint union of its Julia constituents. That is, 
\begin{equation}\label{eq_continuadisj}
f \text { is of disjoint type }\quad \Rightarrow \quad J(f)=\bigcup_{\s\in \Addr(f)}J_\s.
\end{equation}
In particular, all points in $J(f)$ have an external address. However, this is not the case for all functions in class $\B$, as for example occurs when $S(f)\cap J(f)$ is not empty, see \S\ref{sec_signed}.
\end{observation}

For the rest of the section and unless otherwise stated, we assume that for each $f\in \B$, $\Addr(f)$ has been defined with respect to some alphabet of fundamental domains. We require the following properties of Julia constituents.
\begin{thm}[Realisation of addresses {\cite[Theorem 2.5]{lasse_dreadlocks}}]\label{thm_props_Js} Let $f\in\B$. Then, for each external address $\s$, the following holds.
\begin{enumerate}[label=(\alph*)]
\item If $\s$ is admissible, then $J_\s$ contains a closed, unbounded, connected set $X$ on which the iterates of $f$ tend to infinity uniformly.
\item \label{itemb_2.5} If $X_1$ and $X_2$ are unbounded, closed, connected subsets of $J_\s$ with $X_1 \nsubseteq X_2$, then $X_2 \subseteq X_1$ and $f^n \vert_{ X_2} \rightarrow \infty$ uniformly.
\item If $\s$ is bounded, then it is admissible, that is, $J_\s \neq \emptyset$.
\end{enumerate}
\end{thm} 
Note that Julia constituents need not be connected nor closed. Thus, we shall usually and instead, work with the following subsets:
\begin{discussion}[Closed sets in Julia constituents]\label{dis_Jsinfty} For each $\s \in \Addr(f)$, we denote by $J^\infty_\s$ the closure of the union of all closed, unbounded, connected sets $X\subset J_\s$ on which the iterates of $f$ tend to infinity uniformly.
\end{discussion}

Before we continue the study of Julia constituents, we note that an advantage that any $f\in \B$ presents over other transcendental entire maps is that for each $T\in \T_f$, $f\vert_T$ is a covering map that \textit{expands uniformly} the hyperbolic metric that sits on $\C\setminus f(T)$. This well-known fact lies behind many results on functions in $\B$, and goes back to \cite[Lemma~1]{eremenkoclassB}. We denote by $\rho_{\C\setminus \overline{D}}$ the density of the hyperbolic metric in $\C\setminus \overline{D}$, see \cite{beardon_minda} for background.
\begin{prop}[Hyperbolic expansion on tracts]\label{prop_exp_tracts} Let $f\in \B$, fix a domain $D\supset S(f)$ and let $\T_f\defeq f^{-1}(D)$. For each tract $T\in \T_f$, denote by $\unbdd{T}$ the unbounded connected component of $T\setminus \overline{D}$, and let $\unbdd{\T}\defeq \bigcup_{T\in \T_f} \unbdd{T}$. Then, there exists a constant $\Delta>1$ such that 
$$\Vert \Deriv f(z)\Vert_{\C\setminus \overline{D}}\defeq \vert f'(z)\vert \cdot \frac{\rho_{\C\setminus \overline{D}}(f(z))}{\rho_{\C\setminus \overline{D}}(z)}>\Delta$$
for all $z\in\unbdd{\T}$. Moreover, let $\mathcal{A}(D, \delta)$ be an alphabet of fundamental domains. If $S(f)\Subset D'\Subset D$ for a subdomain $D'$, then, for each $R>0$, there exists $C>0$ such that for all $F\in\mathcal{A}(D, \delta)$, the Euclidean and hyperbolic diameter of $\unbdd{F}\cap \D_R$ is less than $C$, where the hyperbolic metric sits in $\C \setminus \overline{D'}$.
\end{prop}
\begin{proof}
For the first part of the statement, see for example \cite[Lemma 5.1]{lasseRidigity}. For the second part, we note that for each fixed $R>0$, only finitely many fundamental domains intersect $\overline{\D}_R$, see \cite[Lemma 2.1]{lasse_dreadlocks} If $F$ is one of those, then $\unbdd{F}\cap \D_R$ is compactly contained in $\C\setminus \overline{D'}$, and so has finite Euclidean and hyperbolic diameter.
\end{proof}

For each set $A\subset \C$, we denote its Hausdorff dimension by $\dim_H A$; see for example \cite[Chapter 3]{Falconer_book_fractals} for definitions. 
\begin{prop}[Hausdorff dimension of non-escaping points with a given address]\label{prop_hausforff}
Let $f\in \B$. Then, for each $\s\in \Addr(f)$, $J_\s^\infty \setminus I(f)$ has Hausdorff dimension zero.
\end{prop}
\begin{remark}
The following proof is essentially the same as the proof of \cite[Proposition~5.9]{lasse_arclike}. Still, for completeness, we include it with the minor modifications that adapt it to our setting.
\end{remark}
\begin{proof}[Proof of Proposition \ref{prop_hausforff}]
For each $\s=F_0 F_1 F_2\dots \in \Addr(f)$ and $n\in \N_{\geq 1}$, we denote 
\[ f_\s^n \defeq f\vert_{\unbdd{F}_{n-1}}\circ f\vert_{\unbdd{F}_{n-2}} \circ \cdots \circ f\vert_{\unbdd{F}_{0}} \quad \text{ and } \quad  f^{-n}_\s\defeq \left(f^{n}_\s\right)^{-1}.\]	
If $z \notin I(f)$, then there is $K>0$ such that $f^n(z) \in \D_K$ for infinitely many $n\geq 0$. Hence, the set of non-escaping points in $J_\s^\infty$ can be written as
$$ J^\infty_{\s}\setminus I(f) = \bigcup_{K=0}^{\infty} \bigcap_{n_0=0}^{\infty} \bigcup_{n= n_0}^{\infty} f_{\s}^{-n}( \unbdd{F}_n \cap \D_K). $$
Since a countable union of sets of Hausdorff dimension zero has Hausdorff dimension zero, it suffices to prove that for each $K>0$, the set
\[ S(K) \defeq \bigcap_{n_0=0}^{\infty} \bigcup_{n= n_0}^{\infty} f_{\s}^{-n}(\unbdd{F}_n \cap \D_K) \]
has Hausdorff dimension zero. Let us fix some arbitrary $K>0$ and let $\mathcal{A}(D, \delta)$ be an alphabet of fundamental domains. Choose a subdomain $D'$ such that $S(f)\Subset D'\Subset D$ and and let $\T_{D'}$ and $\T_{D}$ be the respective corresponding sets of tracts. In particular, $\T_D\Subset \T_{D'}$. Then, by Proposition \ref{prop_exp_tracts}, there exists a constant $C \defeq C (K)$ such that for all $n$, $\unbdd{F}_n\cap \overline{\D}_K$ has diameter at most $C$ in the hyperbolic metric of $\C\setminus \overline{D'}$. Moreover, by the same proposition, there exists a constant $\Delta>1$ such that $\Vert \Deriv f(z)\Vert_{\C\setminus \overline{D'}}\geq \Delta$. In particular, for any domain $S\subset \C\setminus \overline{D'}$ such that $f^{-1}(S)\subset {\C\setminus \overline{D'}}$, $\diam_{\C\setminus \overline{D'}}(f^{-1}(S))\leq \diam_{\C\setminus \overline{D'}}(S) \cdot \Delta$, where $\diam_{\C\setminus \overline{D'}}$ denotes the hyperbolic diameter in $\C\setminus \overline{D'}$. Let us assume that for each fundamental domain $F$, the subset $\unbdd{F}$ is endowed with a hyperbolic metric. Then, since for each $n\geq 1$, the restriction $f\vert_{\unbdd{F}_n}$ is a hyperbolic isometry to a subset of $\C\setminus \overline{D'}$, by Schwarz-Pick Lemma \cite[Lemma 6.4]{beardon_minda}, using the observation above,	\begin{equation}\label{eq_diamSn}
\text{if} \quad S_n\defeq f_{\s}^{-n}(\unbdd{F}_n \cap \D_K) , \quad \text{then} \quad \diam_{\unbdd{F}_0}(S_n) \leq C \cdot \Delta^{-(n-1)},
\end{equation}
where $\diam_{\unbdd{F}_0}$ denotes hyperbolic diameter in $\unbdd{F}_0$. Since $0\notin \T_f$, for each $R\in \N_{\geq 1}$, the Euclidean distance between any point in $\unbdd{F}_0\cap \D_R$ and $\partial \unbdd{F}_0$ is at most $2R$. Thus, by a standard estimate on the hyperbolic metric in a simply-connected domain \cite[Theorem ~8.6]{beardon_minda}, $\rho_{\unbdd{F}_0}(z)\geq 1/(4R)$ for all $z\in \unbdd{F}_0\cap \D_R$. Hence, by \eqref{eq_diamSn}, the Euclidean diameter of $S_n\cap \D_R$ is at most $4R\cdot C \cdot \Delta^{-(n-1)}$. Then, for a fixed $t>0$ and for every $n_0\geq 1 $, the $t$-dimensional Hausdorff measure of $S(K)\cap \D_R$ is bounded from above by
\begin{align*} \liminf_{n_0\to\infty} \sum_{n\geq n_0} \diam(S_n\cap \D_R)^t & \leq
\liminf_{n_0\to\infty} \sum_{n\geq n_0} (4R \cdot C \cdot\Delta^{(-(n-1))})^t \\ &=
(4R C )^t \cdot \lim_{n_0\to\infty} \sum_{n\geq n_0-1} (\Delta^{-t})^n = 0.
\end{align*}
Thus, $\dim_H(S(K)\cap \D_R)\leq t$. Since $t>0$ was arbitrary, $\dim_H(S(K)\cap \D_R)=0$. Using again that a countable union of sets of Hausdorff dimension zero has Hausdorff dimension zero, $\dim_H(S(K))=0$. 
\end{proof}

We shall next see in Theorem \ref{prop_criniferous} that for criniferous functions, the sets defined in \ref{dis_Jsinfty} are either ray tails, or dynamic rays together with their endpoints. In order to prove this theorem, we require some results on continuum theory, that we include here. Recall that a \textit{continuum} $X$ (i.e., a non-empty compact,
connected metric space) is \textit{indecomposable} if it cannot be written as the union of two proper subcontinua of $X$. The \textit{composant of a point} $x\in X$ is the union
of all proper subcontinua of $X$ containing $x$, and a \textit{composant} of $X$ is a maximal set in which any two points lie within some proper subcontinuum of $X$. If $X$ is indecomposable, then there are uncountably many different composants,
every two of which are disjoint, and each of which is connected and dense in $X$, see \cite[Exercise 5.20(a) and Theorem 11.15]{nadler_continuum}.

\begin{thm}[Boundary bumping theorem {\cite[Theorem 5.6]{nadler_continuum}}]\label{thm_bumping} Let $X$ be a continuum and let $E \subsetneq X$ be non-empty. If $K$ is a connected component of $X\setminus E$, then $\overline{K} \cap \partial E \neq \emptyset$.
\end{thm}
We will moreover make use of the following result in order to show that the accumulation set of a dynamic ray is an
indecomposable continuum:
\begin{thm}[Curry {\cite{Curry}}]\label{curry} Suppose that $\gamma$ is a ray, i.e. a continuous injective image of $[0,1)$, and let $\Lambda(\gamma)$ denote its accumulation set. If $\Lambda(\gamma)$ has topological dimension one, does not separate the Riemann sphere into infinitely many components and contains $\gamma$, then $\Lambda(\gamma)$ is an indecomposable continuum.
\end{thm}

\begin{thm}[Criniferous functions in $\B$] \label{prop_criniferous}
Let $f\in \B$ be criniferous. Then, for each $\s\in \Addr(f)$, $J^\infty_\s$ is either a ray tail, or a dynamic ray together with its endpoint. In particular,
	\begin{equation}\label{eq_IfinOrbminus}
	I(f) \subset \bigcup_{n\geq 0} f^{-n}\Bigg( \bigcup_{\s \in \Addr(f)} J^\infty_\s \Bigg).
	\end{equation}
\end{thm}
\begin{proof}
Fix $\s \in \Addr(f)$ and let us choose any $z \in J^\infty_\s \cap I(f)$. Since $f$ is criniferous, there exists $N\geq 0$ so that $f^N(z)$ is the endpoint of a ray tail $\gamma$. Then, since, by definition, ray tails escape uniformly to infinity, there exists a constant $M\defeq M(\gamma)\in \N$ such that $f^m(\gamma)\subset \T_f$ for all $m\geq M$, which in particular implies that $f^m(\gamma)$ must be totally contained in a fundamental domain for each $m\geq M$. More specifically, since $f^{N+m}(z)\in f^m(\gamma)$ for any $m \geq M$, the curve $f^m(\gamma)$ belongs to the same fundamental domain $f^{N+m}(z)$ does, which in turn is determined by the external address $\s$. Hence, all points in $f^m(\gamma)$ have external address $\sigma^{m+N}(\s)$, and in particular, $f^{M}(\gamma) \subset J^\infty_{\sigma^{N+M}(\s)}$. Moreover, by \eqref{eq_inclJs}, it also holds that $f^{N+M}(J^{\infty}_\s)\subseteq J^{\infty}_{\sigma^{N+M}(\s)}$. Since the restriction of $f$ to any Julia constituent is injective, as they all lie outside a Jordan domain that contains $S(f)$, by definition of $J^{\infty}_\s$, $f^M(\gamma)\subseteq f^{N+M}(J^{\infty}_\s)$. Hence, the curve in the $(N+M)$-th preimage of $f^{M}(\gamma)$ that intersects $J^{\infty}_\s$ must be a ray tail with endpoint $z$, that we denote by $\gamma_z$. That is, $\gamma_z\defeq f^{-N-M}(f^M(\gamma)) \cap J^{\infty}_\s$.
	
If $z,w \in J^\infty_\s \cap I(f)$, then by Theorem \ref{thm_props_Js}\ref{itemb_2.5} either $\gamma_w\subset \gamma_z$ or $\gamma_z \subset \gamma_w$, and thus, these curves are totally ordered by inclusion. Hence, 
$$\gamma\defeq \bigcup_{z \in I(f)\cap J^\infty_\s} \gamma_z$$
is a maximal injective curve in $I(f)$ that escapes uniformly to infinity, and in particular, $\gamma=I(f)\cap J^\infty_\s$. If $J^\infty_\s\subset I(f)$, then $\gamma=J^\infty_\s$ is a ray tail and we are done. Otherwise, let us parametrize $\gamma: (0,\infty)\rightarrow \C$, and denote by $\Lambda(\gamma)$ the accumulation set of $\gamma(t)$ as $t \rightarrow 0$. In particular, since $J^\infty_\s$ is closed, $J^\infty_\s\supseteq \gamma\cup \Lambda(\gamma)$. Let us compactify $J^\infty_\s$ by adding infinity, i.e., $\widehat{J}_\s \defeq J^\infty_\s\cup \{\infty\}$. If $w\in J^\infty_\s \setminus \gamma$, then by the boundary bumping theorem (Theorem~\ref{thm_bumping}), if $K$ is the connected component of $\widehat{J}_\s \setminus \gamma$ containing $w$, then $\overline{K} \cap \gamma\neq \emptyset$. But then, since $K \subset \widehat{J}_\s\setminus I(f)$, by Proposition \ref{prop_hausforff}, the set $K$ must be the singleton $\{w\}$, and thus $w\in \Lambda(\gamma)$. Therefore, we have shown that 
$$J^\infty_\s = \gamma\cup \Lambda(\gamma).$$

Consequently, it suffices for our purposes to study the set $ \Lambda(\gamma)$. First, we note that for ever $t>0$ of $\gamma$, there are pieces of other dynamic rays of $f$ accumulating uniformly \textit{from above} and \textit{from below} on $\gamma_{[t, \infty)}$. This follows from a well-known argument\footnote{Compare to the proof of \cite[Corollary 6.7]{lasse_dreadlocks} or \cite[Corollary 6.9]{Schleicher_Zimmer_exp}.}, that we sketch here. Fix $z\defeq \gamma(t^\ast)$ for some $t^\ast\geq t$. For each fundamental domain $F$, there exists a pair of fundamental domains $F^-$ and $F^+$, that are respectively the immediate predecessor and successor of $F$ in the cyclic order at infinity of fundamental domains, see \ref{dis_cyclic}. In particular, $F^-, F,F^+$ lie in the same tract. Then, for each $k\geq 0$, consider the external address $\s^+_k$ that equals $\s$ except on its $k$-th entry, which is $F_k^+$ instead of $F_k$; and similarly, $\s^-_k$ equals $\s$ except that its $k$-th entry is $F_k^-$. Then, using the expansion property that $f$ has in tracts, (Proposition \ref{prop_exp_tracts}), one can see by mapping forward and pulling back through appropriate inverse branches, that $J_{\s_k^+}(t^\ast)\rightarrow z$ from above, and $J_{\s_k^-}(t^\ast)\rightarrow z$ from below as $k\rightarrow \infty$.

This implies that if $\gamma(t_0)\in \Lambda(\gamma)$ for some $t_0>0$, then the curve $\gamma$ must also accumulate on $\gamma([t, t_0])$ for all $0<t<t_0$. Hence, by letting $t\rightarrow 0$, we see that $\gamma_{(0, t_0]}\subset \Lambda(\gamma)$. Let $t_0$ be any potential such that $\gamma(t_0)\in \Lambda(\gamma)$. Then, by Theorem \ref{curry}, $\Lambda(\gamma_{(0, t_0]})$ must be an indecomposable continuum. Recall that this means that all composants of $\Lambda(\gamma_{(0, t_0]})$ are pairwise disjoint and dense in $\Lambda(\gamma_{(0, t_0]})$, and in particular, since their closures must contain $\gamma_{(0, t_0]}$, these composants must be non-trivial. However, this would contradict Proposition~\ref{prop_hausforff}, and thus, $\Lambda(\gamma)$ must be a singleton, namely the landing point of $\gamma$. Hence, $J_\s^\infty\setminus I(f)$ is the landing point of the dynamic ray $\gamma$.

Finally, \eqref{eq_IfinOrbminus} follows noting that if $z \in I(f)$, then, arguing as before, there is $N \defeq N(z)\in \N$ and a ray tail $\gamma \in \gamma \subset J^{\infty}_{\ultau}$ for some $\ultau \in \Addr(f)$.
\end{proof}

\begin{discussion}[Cyclic order and topology in $\Addr(f)$] \label{dis_cyclic}
Let $f\in \B$, and let $\Addr(f)$ be a set of external addresses defined from an alphabet of fundamental domains $\mathcal{A}(D,\delta)$. There is a natural \emph{cyclic order} on $\mathcal{A}(D,\delta)$, together with the curve $\delta$: if $X, Y, Z\in \mathcal{A}(D,\delta) \cup \{\delta\}$, then we write
\begin{equation}\label{eq_orderinfty}
[X,Y,Z]_{\infty} \: \: \Leftrightarrow \: \: Y \text{ tends to infinity between }X\text{ and } Z \text{ in positive orientation.}
\footnote{See \cite[13. Appendix]{lasse_dreadlocks} for details on the existence of a cyclic order on any pairwise disjoint collection of unbounded, closed, connected subsets of $\C$, none of which separates the plane.}
\end{equation}
From this cyclic order, it is possible to define a \textit{lexicographical order} on the set $\Addr(f)$: we can define a linear order on $\mathcal{A}(D, \delta)$ by ``cutting'' $\delta$ the following way: $$F < \tilde{F} \quad \text{ if and only if } \quad [\delta, F, \tilde{F}]_\infty.$$ 
Then, $\mathcal{A}(D, \delta)$ becomes totally ordered, and this order gives rise to a lexicographical order ``$<_{_\ell}$'' on external addresses, defined in the usual sense. In turn, $\Addr(f)$ becomes a totally ordered set, and hence we can define a cyclic order induced by $<_{_\ell}$ the usual way:
\begin{equation}\label{eq_cyclicell}
[\s,\underline{\alpha},\underline{\tau}]_{\ell} \quad \text{if and only if} \quad \s <_{_\ell}\underline{\alpha}<_{_\ell}\underline{\tau} \quad \text{ or } \quad \underline{\alpha}<_{_\ell} \underline{\tau} <_{_\ell} \s \quad \text{ or } \quad \underline{\tau} <_{_\ell} \s <_{_\ell} \underline{\alpha}.
\end{equation}
The cyclic order on addresses specified in \eqref{eq_cyclicorder}, allows us to provide the set $\Addr(f)$ with a topology: given two different elements $\s,\ultau \in \Addr(f)$, we define the \textit{open interval} from $\s$ to $\ultau$, denoted by $(\s,\ultau)$, as the set of all addresses $\ul{\alpha}\in \Addr(f)$ such that $[\s, \ul{\alpha}, \ultau]_{\ell}$. The collection of all these open intervals forms a base for the \textit{cyclic order topology.}\footnote{In particular, the open sets in this topology happen to be exactly those ones which are open in every compatible linear order.}
\end{discussion}

\begin{remark}
Unless otherwise stated, from now on and when working with external addresses, we will assume that the set $\Addr(f)$ has been endowed with the cyclic order topology.
\end{remark}

\begin{observation}
This cyclic order on addresses in \eqref{eq_cyclicell} agrees with the cyclic order at infinity on $\{J^\infty_\s\}_{\s\in \Addr(f)}$. That is,
\begin{equation}\label{eq_cyclicorder}
[\s,\underline{\alpha},\underline{\tau}]_\ell \quad \text{if and only if} \quad[J^\infty_\s,J^\infty_{\underline{\alpha}},J^\infty_{\underline{\tau}}]_{\infty}.
\end{equation}
\end{observation}
\begin{proof}
It follows from the cyclicity axiom of ternary relations (that is, if $[a,b,c]$, then $[b,c,a]$), together with the following claim: for any pair $\s, \ul{\alpha}\in \Addr(f)$ such that $\s<_{_\ell}\underline{\alpha}$, it holds $[\delta, J^\infty_\s, J^\infty_{\underline{\alpha}}]_{\infty}$, where we have considered the cyclic order at infinity of all Julia constituents together with the curve $\delta$. 

Indeed, suppose that $\s\defeq F^{s}_0F^{s}_1\ldots$, and $\underline{\alpha}\defeq F^{\alpha}_0F^{\alpha}_1\ldots$, and that $\s$ and $\underline{\alpha}$ first differ in their $k$-th entry for some $k\in \N$. That is, $F^{s}_i=F^{\alpha}_i$ for all $i<k$, and $F^{s}_k\neq F^{\alpha}_k.$ Note that $\s<_{_\ell}\underline{\alpha}$ holds if and only if $f^i(J^\infty_\s), f^i(J^\infty_{\underline{\alpha}}) \subset F^s_i$ for all $i<k$, and $\unbdd{F}^{s}_k\neq \unbdd{F}^{\alpha}_k.$ Equivalently, $\s<_{_\ell}\underline{\alpha}$ if and only if $f^i(J^\infty_\s), f^i(J^\infty_{\underline{\alpha}}) \subset \unbdd{F}^s_i$ for all $i<k$ and $[\delta, f^k(J^\infty_\s), f^k(J^\infty_{\underline{\alpha}})]_{\infty}.$ But then, since $f$ acts as a conformal isomorphism from each fundamental domain to $\mathcal{W}$, in particular $f$ preserves the cyclic order at infinity of Julia constituents. Thus,
\begin{equation*}
\begin{split}
[\delta, f^k(J^\infty_\s), f^k(J^\infty_{\underline{\alpha}})]_{\infty}&\Longleftrightarrow [\delta, f^{k-1}(J^\infty_\s),f^{k-1}(J^\infty_{\underline{\alpha}})]_{\infty}\Longleftrightarrow \cdots \Longleftrightarrow[\delta, f(J^\infty_\s),f(J^\infty_{\underline{\alpha}})]_{\infty}\\
&\Longleftrightarrow [\delta, J^\infty_\s, J^\infty_{\underline{\alpha}}]_{\infty} ,
\end{split}
\end{equation*}
and the claim follows.
\end{proof}

We would like to point out to the reader that providing $\Addr(f)$ with a topological structure allows us to use the notion of \textit{convergence} of external addresses. In particular, for disjoint type functions, convergence of addresses is closely related to how the corresponding unbounded components of Julia constituents \textit{accumulate} on the plane. More specifically, let $f\in \B$ be of disjoint type, and for every $w\in J(f)$, denote by $\text{add}(w)$ its external address. For a sequence of points $\{z_k\}_k$ in $J(f)$, 
\begin{equation}\label{eq_sktos}
\text{if} \quad z_k\rightarrow z , \quad \text{ then } \quad \text{add}(z_k)\rightarrow \text{add}(z) \quad \text {as } k\rightarrow \infty,
\end{equation}
which is a consequence of the characterization of disjoint type functions in Proposition \ref{prop_charac_disjoint} and expansion from Proposition \ref{prop_exp_tracts}. See the proof of Theorem \ref{prop_criniferous} for details on a similar argument.
\section{Signed addresses for criniferous functions}\label{sec_signed}
We have defined in Section \ref{sec_symbolic}  external addresses for each $f\in \B$, and in particular, for some but not all points in $I(f)$, see Observation \ref{obs_addrz}. In this section, given some additional assumptions on $f$, we introduce a new form of address, generalizing Definition \ref{def_extaddr}, so that all points in $I(f)$ have (at least) one of these new addresses, that we call \textit{signed addresses}. More specifically, our aim is to define signed addresses for criniferous functions in class $\B$ that do not have asymptotic values in their Julia sets. In particular, we consider functions that might contain escaping critical values. Hence, in a very rough sense, any sensible analogue of Julia constituents, defined in \eqref{eq_Js}, that is, satisfying properties \eqref{eq_inclJs} and \eqref{eq_sktos}, would have to consider ``bifurcations'' or ``splitting'' at critical points. To illustrate this, we study the map $f=\cosh$.
\begin{example}[Signed addresses for $\cosh$]\label{example_cosh}
For $f=\cosh$, $S(f)=\CV(f)=\{-1, 1\}$, and so we can define tracts and fundamental domains for $f$ using a disc $D \supset \{-1, 1\}$ and letting $\delta$ be the piece of positive imaginary axis connecting $\partial D$ to infinity. For this choice of $D$ and $\delta$, each fundamental domain of $f$ is contained in one of the horizontal half-strips
\begin{equation}\label{fund_cosh}
\begin{split}
S_{n_L}&\defeq\{z: \Rea z<0, \Ima z\in ((n-1/2)\pi ,(n+3/2)\pi)\} \quad \text{ or } \\
S_{n_R}&\defeq\{z: \Rea z>0, \Ima z\in ((n-3/2)\pi,(n+1/2)\pi)\};
\end{split}
\end{equation}
see \cite[\S5]{mio_cosine} for more details. Thus, if we label the fundamental domains of $f$ by the sub-index of the strip they belong to, is easy to see that for the external address $\overline{0_R}$, $J_{\overline{0_R}} \subset \R^+$. Moreover, $J_{\overline{0_R}}$, and in fact $\R^+$, are ray tails. 

We shall extend the curve $J_{\overline{0_R}}$ in different ways so that the extensions are still ray tails. By \eqref{eq_inclJs}, $f(J_{\overline{0_R}}) \subset J_{\sigma(\overline{0_R})}=J_{\overline{0_R}}$ and thus, there exists a preimage of $J_{\overline{0_R}}$ that contains $J_{\overline{0_R}}$. Let us denote that preimage by $J^0_{\overline{0_R}}$. If $J^0_{\overline{0_R}}\cap \Crit(f)=\emptyset$, then $J^0_{\overline{0_R}}$ is by definition a ray tail. In that case, we denote by $J^1_{\overline{0_R}}$ the preimage of $J^0_{\overline{0_R}}$ that contains $J^0_{\overline{0_R}}$. We can iterate this process until for some $n\in \N$, a preimage $\beta$ of $J^n_{\overline{0_R}}$ contains the critical point $0$. Then, $\beta$ is no longer a ray tail, but instead, $\beta \setminus [0, -i\pi/2]$ and $\beta \setminus [0, i\pi/2]$, where $[0, \pm i\pi/2]$ are vertical segments in the imaginary axis, are ray tails. Thus, a choice has to be made on how to define $J^{n+1}_{\overline{0_R}}$. However, if we extend in the same fashion other Julia constituents $J_{\s_i}$ for addresses $\s_i$ ``sufficiently close'' to $\overline{0_R}$, a more careful analysis would show that whenever $\s_i\rightarrow \overline{0_R}$ ``from above'' (see \ref{dis_cyclic}), $J^{n+1}_{\s_i} \rightarrow [0, i\pi/2]\cup \R^+$, and whenever $\s_i\rightarrow \overline{0_R}$ ``from below'', $J_{\s_i} \rightarrow [0, -i\pi/2]\cup \R^+$. Hence, for an analogue of property \eqref{eq_sktos} to hold, we would have to extend $J_{\overline{0_R}}$ to include both of those two segments. But then, such extension would not be a ray tail. We resolve this obstacle by considering two copies of $\Addr(f)$ indexed by $\{-,+\}$ and defining two ray tails $J^{n+1}_{(\overline{0_R}, +)}\defeq [0, i\pi/2]\cup \R^+$ and $J^{n+1}_{(\overline{0_R}, -)} \defeq [0, -i\pi/2]\cup \R^+$. By providing $\Addr(f) \times \{-,+\}$ with the ``right'' topology, an expression similar to \eqref{eq_sktos} holds for the elements in $\Addr(f) \times \{-,+\}$, that we call signed addresses. 
\end{example}
\noindent We now formalize these ideas with more generality:
\begin{discussion}[Space of signed addresses]\label{discussion_signed_addr}
Let $f\in \B$, and let $\Addr(f)$ be a set of admissible external addresses. Let us consider the set
$$\Addr(f)_\pm \defeq \Addr(f) \times \{-,+\},$$
that we endow with a topology: let $<_\ell$ be the lexicographical order in $\Addr(f)$ defined in \ref{dis_cyclic}, and let us give the set $\lbrace -,+\rbrace$ the order $\lbrace -\rbrace \prec\lbrace +\rbrace$. Define the linear order
\begin{equation} \label{eq_linear_addr}
(\ul{s}, \ast )<_{_A} (\ul{\tau}, \star) \qquad \text{ if and only if } \qquad \ul{s} <_{_\ell} \ul{\tau} \quad \text{ or } \quad \ul{s} =_{_\ell} \ul{\tau}\: \text{ and } \: \ast \prec \star,
\end{equation}
where the symbols ``$\ast, \star$'' denote generic elements of $\lbrace-, +\rbrace.$ 
This linear order gives rise to a cyclic order: for $a,x,b \in \Addr(f)_\pm$,
\begin{equation} \label{eq_orderAfpm}
[a,x,b]_{_A} \quad \text{if and only if} \quad a<_{_A} x<_{_A} b \quad \text{ or } \quad x <_{_A}b <_{_A} a \quad \text{ or }\quad b <_{_A} a <_{_A} x.
\end{equation}
In turn, this cyclic order allows us to define a \textit{cyclic order topology} $\tau_A$ in $\Addr(f)_\pm.$
\end{discussion}
\begin{defn}[Signed external addresses for criniferous functions]\label{defn_signedaddr} Let $f\in \B$ be a criniferous function and let $(\Addr(f)_\pm, \tau_A)$ be the corresponding topological space defined according to \ref{discussion_signed_addr}. A \emph{signed (external) address} for $f$ is any element of $\Addr(f)_\pm$. 
\end{defn}
For each criniferous function $f\in \B$ such that $J(f)\cap \AV(f)=\emptyset$, we aim to define signed external addresses for all points in $I(f)$ by extending subcurves of Julia constituents in a systematic way, as described in Example \ref{example_cosh}. In order to do so, we start by settling which extensions will be allowed at critical points, and defining \textit{canonical tails} as curves in the escaping set that agree with the criterion established.

Recall that the \emph{local degree} of $f$ at a point $z_0\in \C$, denoted by $\deg(f,z_0)$, is the unique integer $n\geq 1$ 
such that the local power series development of $f$ is of the form
\begin{equation*}
f(z)=f(z_0) + a_n (z-z_0)^n + \text{(higher terms)},
\end{equation*}
where $a_n\neq 0$.
\begin{discussion}[Extensions at critical points]\label{dis_canonical} Let $f\in \B$ and let $\delta$ be either a ray tail, or a dynamic ray (possibly together with its endpoint) such that $\delta \cap \AV(f)= \emptyset$ and $\delta \cap \CV(f)\neq \emptyset$. Let $\beta$ be a connected component of $f^{-1}(\delta)$ such that $\beta\cap \Crit(f)\neq \emptyset$. Then, each critical point $c\in \beta$ is the endpoint of $2\deg(f, c)$ curves in $\beta\setminus \Crit(f)$. We denote the set of all such curves by $\mathcal{L}(c)$ and note that topologically, each of them is a radial line from $c$. See Figure \ref{fig:canonical_tails}. For each $\alpha \in \mathcal{L}(c)$, let $\alpha^-, \alpha^+ \in \mathcal{L}(c)$ be the respective successor and predecessor curves of $\alpha$ with respect to the anticlockwise circular order of (topological) radial segments in $\mathcal{L}(c)$. Note that by construction, $f(\alpha^- \bm{\cdot} \{c\}\bm{\cdot} \alpha)$ and $f( \alpha^+ \bm{\cdot} \{c\}\bm{\cdot} \alpha)$ are mapped univalently to a subcurve of $\delta$.
\end{discussion}

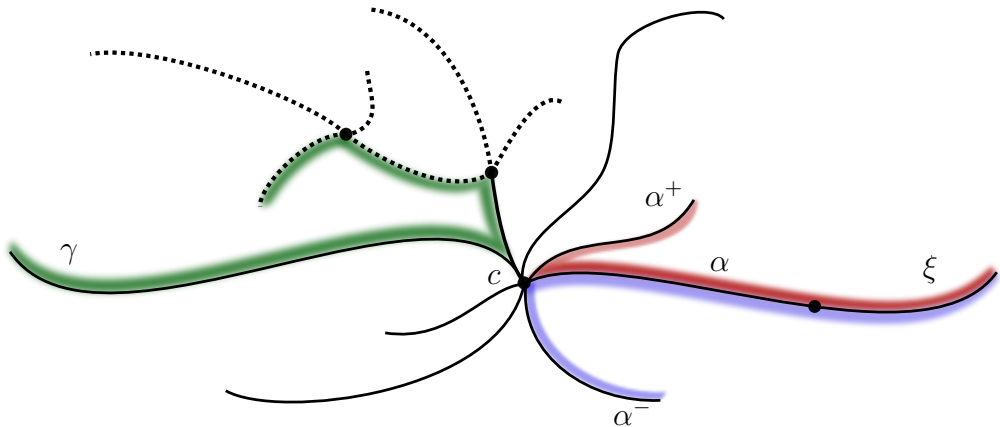
\begin{figure}[h]
	\centering
\begingroup%
  \makeatletter%
  \providecommand\color[2][]{%
    \errmessage{(Inkscape) Color is used for the text in Inkscape, but the package 'color.sty' is not loaded}%
    \renewcommand\color[2][]{}%
  }%
  \providecommand\transparent[1]{%
    \errmessage{(Inkscape) Transparency is used (non-zero) for the text in Inkscape, but the package 'transparent.sty' is not loaded}%
    \renewcommand\transparent[1]{}%
  }%
  \providecommand\rotatebox[2]{#2}%
  \newcommand*\fsize{\dimexpr\f@size pt\relax}%
  \newcommand*\lineheight[1]{\fontsize{\fsize}{#1\fsize}\selectfont}%
  \ifx\svgwidth\undefined%
    \setlength{\unitlength}{425.19685039bp}%
    \ifx\svgscale\undefined%
      \relax%
    \else%
      \setlength{\unitlength}{\unitlength * \real{\svgscale}}%
    \fi%
  \else%
    \setlength{\unitlength}{\svgwidth}%
  \fi%
  \global\let\svgwidth\undefined%
  \global\let\svgscale\undefined%
  \makeatother%
  \begin{picture}(1,0.4)%
    \lineheight{1}%
    \setlength\tabcolsep{0pt}%
    \put(0,0){\includegraphics[width=\unitlength,page=1]{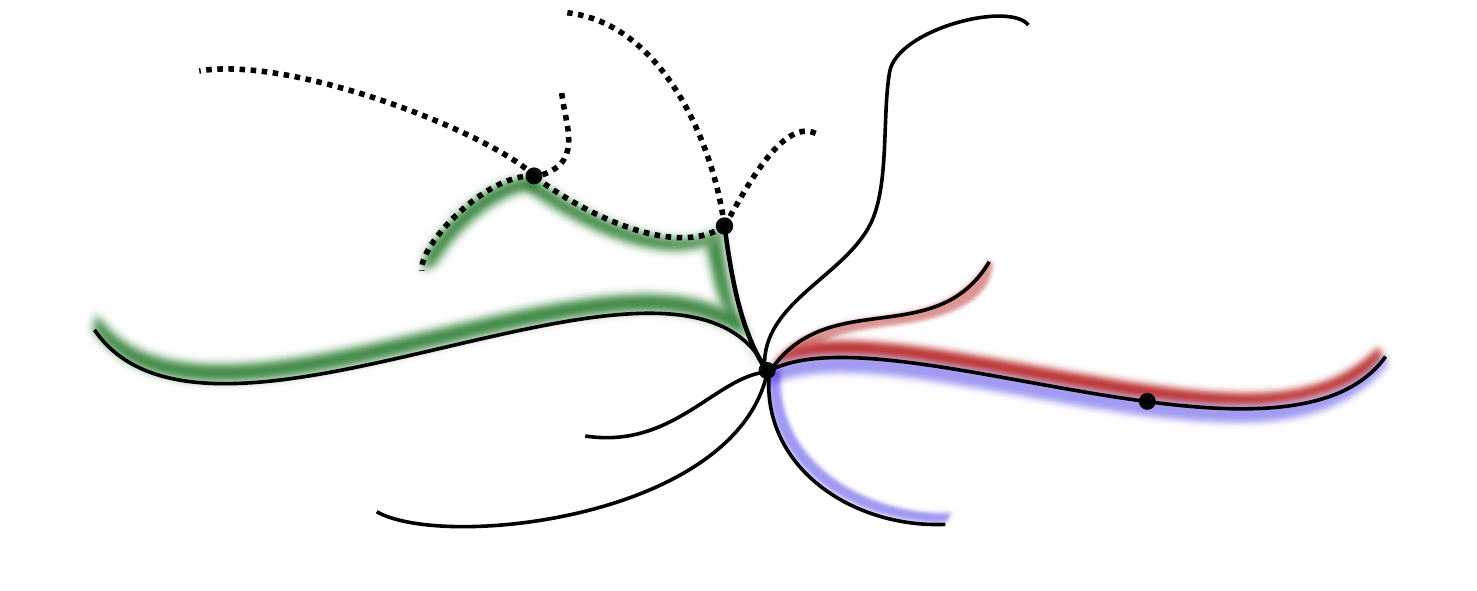}}%
    \put(0.48684461,0.14845912){\color[rgb]{0,0,0}\makebox(0,0)[lt]{\lineheight{1.25}\smash{\begin{tabular}[t]{l}$c$\end{tabular}}}}%
    \put(0.68407853,0.15962049){\color[rgb]{0,0,0}\makebox(0,0)[lt]{\lineheight{1.25}\smash{\begin{tabular}[t]{l}$\alpha$\end{tabular}}}}%
    \put(0.87291331,0.15446997){\color[rgb]{0,0,0}\makebox(0,0)[lt]{\lineheight{1.25}\smash{\begin{tabular}[t]{l}$\xi$\end{tabular}}}}%
    \put(0.62672088,0.21727364){\color[rgb]{0,0,0}\makebox(0,0)[lt]{\lineheight{1.25}\smash{\begin{tabular}[t]{l}$\alpha^+$\end{tabular}}}}%
    \put(0.59857033,0.02322734){\color[rgb]{0,0,0}\makebox(0,0)[lt]{\lineheight{1.25}\smash{\begin{tabular}[t]{l}$\alpha^-$\end{tabular}}}}%
    \put(0.10842542,0.17097423){\color[rgb]{0,0,0}\makebox(0,0)[lt]{\lineheight{1.25}\smash{\begin{tabular}[t]{l}$\gamma$\end{tabular}}}}%
  \end{picture}%
\endgroup
	\caption{Definition of bristles and canonical tails. In the picture, critical points are represented by black dots, and curves in $\mathcal{L}(c)$ with continuous strokes. The curve $\gamma$, shown in green, is a canonical tail, and in particular a left-extended curve. The curves $\alpha^+$ and $\alpha^-$ are the respective right and left bristles of the curve $\alpha$, that has $c$ as an endpoint.}
\label{fig:canonical_tails}
\end{figure}

\begin{defn}[Canonical tails and rays] \label{def_leftrightextensions}Following \ref{dis_canonical}, we define:
\begin{itemize}[noitemsep,wide=0pt, leftmargin=\dimexpr\labelwidth + 2\labelsep\relax]
	\item The curves $\alpha^-$ and $\alpha^+$ are the respective \emph{left} and \emph{right bristles} of $\alpha$.
	\item For any curve $\xi \subset \beta$ such that $\xi \cap \mathcal{L}(c)=\alpha$ for some $\alpha \in \mathcal{L}(c)$, the concatenations
	$$ \alpha^- \bm{\cdot} \{c\}\bm{\cdot} \xi \quad \text{ and } \quad \alpha^+ \bm{\cdot} \{c\}\bm{\cdot} \xi $$	
	are the respective \emph{left} and \emph{right extensions of} $\xi$ at $c$.
	\item  Let $\lambda \subset \beta$ be an unbounded curve with finite endpoint $c_0$, and suppose that $\lambda\cap \Crit(f)=\{c_1, \ldots, c_n\}$ for some $n\in \N$, where the points $c_i$ are ordered from smallest to largest potential in 
	$\lambda$. For each $i$, let $\lambda_i$ be the unbounded curve in $\lambda\setminus \{c_i\}$. If for all $0\leq i\leq n-1$, $\lambda_i$ is a right (resp. left) extension of $\lambda_{i+1}$ at $c_{i+1}$, then we say that $\lambda$ is a \textit{right-extended curve} (resp. \textit{left-extended curve}).
\item If $\gamma$ is ray tail (resp. dynamic ray possibly with its endpoint) for which for all $n\geq 0$ such that $\Crit(f)\cap f^n(\gamma)\neq \emptyset$, the curve $f^n(\gamma)$ is either a right-extended or left-extended curve, all of the same type, then we say that $\gamma$ is a \textit{canonical tail} (resp. \textit{canonical ray}) \textit{of $f$.}
\end{itemize}
\end{defn}

\begin{remark}
If $\gamma\subseteq J^\infty_\s$ is a ray tail (resp. dynamic ray) for some $\s \in \Addr(f)$, then $\gamma$ is a canonical tail (resp. ray), since by definition of Julia constituents, $\Orb^+(J_\s^\infty) \cap \Crit(f)=\emptyset.$ 
\end{remark}

In the forthcoming Theorem \ref{thm_signed}, for certain criniferous functions, we will establish a correspondence between canonical tails and signed addresses. We achieve this by extending each of the curves $J^\infty_\s$ in two ways, so that all extensions are canonical curves, and so that all points in $I(f)$ belong to at least one canonical curve. In certain cases, rather than extending directly the curve $J^\infty_\s$, for technical reasons, it is more convenient to extend some unbounded subcurve of $J^\infty_\s$. The next definition establishes which conditions the mentioned subcurves must fulfil.

\begin{defn}(Initial configuration of tails)\label{def_initial_conf} Let $f\in \B$ and suppose that for each $\s \in \Addr(f)$, there exists a curve $\gamma^0_\s\subset J^\infty_\s$ that is either a ray tail, or a dynamic ray possibly with its endpoint. The set of curves $\{\gamma^0_\s\}_{\s \in \Addr(f)}$ is a \emph{valid initial configuration} for $f$ if for each $\s\in \Addr(f)$, $f(\gamma^0_\s) \subset \gamma^0_{\sigma(\s)}$ and
\begin{equation}\label{eq_S0}
I(f) \subset \Orb^{-}\left( \bigcup_{\s \in \Addr(f)} \gamma^0_\s \right) \eqdef \mathcal{S}.
\end{equation}	
\end{defn}	 

\begin{observation}[Existence of initial configuration equivalent to criniferous]\label{obs_config_crini} Note that if for a function $f\in \B$ there exists a valid initial configuration, by \eqref{eq_S0}, $f$ is criniferous. Conversely, if $f\in \B$ is criniferous, by Theorem \ref{prop_criniferous}, $\{J^\infty_\s\}_{\s \in \Addr(f)}$ is a valid initial configuration for $f$. Moreover, note that all curves in a valid initial configuration are canonical and pairwise disjoint, as Julia constituents are.
\end{observation}

\begin{thm}[Indexed canonical tails]\label{thm_signed} Let $f\in \B$ be criniferous with $\AV(f)\cap J(f)=\emptyset$, and let $\{\gamma^0_\s\}_{\s \in \Addr(f)}$ be a valid initial configuration for $f$.  Then, for each $n\in \N$ and $(\s, \ast) \in \Addr(f)_\pm$, there exists a curve $\gamma^n_{(\s,\ast)}$, that is either a canonical tail or dynamic ray with possibly its endpoint, with the following properties:
\begin{enumerate}[label=(\alph*)]
\item $\gamma^0_{(\s,-)}\defeq\gamma^0_{(\s,+)}\defeq\gamma^0_\s\subset J^\infty_\s$; 
\item for all $n\geq 1$, $\gamma^{n-1}_{(\s, \ast)} \subseteq \gamma^{n}_{(\s, \ast)}$ and $f \colon \gamma^{n}_{(\s, \ast)} \to \gamma^{n-1}_{(\sigma(\s), \ast)}$ is a bijection.
\label{item:tails_bijection}
\end{enumerate}
In particular, if $\mathcal{S}$ is the set from \eqref{eq_S0} and, for each $(\s, \ast) \in \Addr(f)_\pm$, we define the \emph{$\Gamma$-curve}
$\Gamma(\s, \ast)\defeq \bigcup_{n\geq 0}\gamma^n_{(\s, \ast)},$
then $\mathcal{S}= \bigcup_{(\s, \ast)\in \Addr(f)_\pm}\Gamma(\s, \ast)$.
\end{thm}

\begin{proof}
Without loss of generality and for clarity of exposition, we assume that all curves in the given initial configuration are canonical tails, since our arguments work exactly the same way if any curve is a dynamic ray (possibly with its endpoint). We construct canonical tails inductively on $n$ and simultaneously for all elements in $\Addr(f)_\pm$. Let $n=1$ and choose any $(\s, \ast) \in \Addr(f)_\pm$. Since by assumption $f(\gamma^0_\s)\subset \gamma^0_{\sigma(\s)}$, there exists a connected component $\beta$ of $f^{-1}(\gamma^0_{(\sigma(\s), \ast)})$ such that $\gamma^0_{(\s, \ast)} \subseteq \beta .$ Define 
$$ \gamma^1_{(\s, \ast)}\defeq \beta, $$ 
which is a canonical tail since it is a preimage of the canonical tail $\gamma^0_{(\sigma(\s), \ast)}$, that by definition does not contain any singular values, and hence $\gamma^1_{(\s, \ast)}\cap \Crit(f)=\emptyset.$ Note that $\gamma^1_{(\s, -)}=\gamma^1_{(\s, +)}$, and so the curve can be regarded both as left-extended or right-extended. However, for the purpose of our inductive argument, we regard $\gamma^1_{(\s, -)}$ as a left-extended curve, and $\gamma^1_{(\s, +)}$ is a right-extended curve.

Suppose that \ref{item:tails_bijection} has been proved for some $n\in \N$ and all elements in $\Addr(f)_\pm$. We shall see that it holds for $n+1$. By the inductive hypothesis, for each $(\s, \ast)\in \Addr(f)_\pm$, both $\gamma^n_{(\s, \ast)}$ and $\gamma^n_{(\sigma(\s), \ast)}$ are well-defined canonical tails, and $f(\gamma^n_{(\s, \ast)})=\gamma^{n-1}_{(\sigma(\s), \ast)}\subset \gamma^n_{(\sigma(\s), \ast)}$. Moreover, $\gamma^n_{(\sigma(\s), -)}$ must be a left-extended curve, and $\gamma^n_{(\sigma(\s), +)}$ a right-extended curve. Let $\beta$ be the component of $f^{-1}(\gamma^n_{(\sigma(\s), \ast)})$ that contains $\gamma^n_{(\s, \ast)}$. If $(\beta\setminus \gamma^n_{(\s, \ast)}) \cap \Crit(f)=\emptyset$, then we denote
$$\gamma^{n+1}_{(\s, \ast)}\defeq\beta,$$
which is a canonical tail by the same argument as before. Otherwise, $\gamma^n_{(\s, \ast)}$ must be contained in a unique connected component $L_1$ of $\beta\setminus (\Crit(f)\setminus \gamma^n_{(\s, \ast)}).$ In particular, $L_1\setminus \gamma^n_{(\s, \ast)}$ does not contain any critical points, but can be extended to contain a critical point $c_1 \in \beta$ as finite endpoint. Hence, since by the inductive hypothesis $f\vert_{ \gamma^{n}_{(\s, \ast)}}$ maps bijectively to $\gamma^{n-1}_{(\sigma(\s), \ast)}$, $f$ maps the curve $\{c_1\}\bm{\cdot}L_1$ univalently to $\gamma^{n}_{(\sigma(\s), \ast)}$. If $f(\{c_1\} \bm{\cdot} L_1)=\gamma^n_{(\sigma(\s), \ast)}$, then we define 
$$\gamma^{n+1}_{(\s, \ast)}\defeq \{c_1\} \bm{\cdot} L_1, $$
and the claim follows. If, on the contrary, $f(\{c_1\} \bm{\cdot} L_1)\subsetneq \gamma^n_{(\sigma(\s), \ast)}$, then $L_1$ must be a curve containing a unique element of $\mathcal{L}(c_1)$. Thus, following Definition \ref{def_leftrightextensions}, we define $L_2$ as the respective right or left extension of $L_1$ at $c_1$, according to whether $\ast=+$ or $\ast=-$. That is, if $\alpha^-$ and $\alpha^+$ are the respective left and right bristles of $L_1$ at $c_1$, then we define
\begin{equation*}
L_2\defeq \renewcommand{\arraystretch}{1.5}\left\{\begin{array}{@{}l@{\quad}l@{}}
\alpha^+ \bm{\cdot} \{c_1\} \bm{\cdot} L_1 & \text{if } \ast=+, \quad \text{ or} \\
\alpha^- \bm{\cdot} \{c_1\} \bm{\cdot} L_1 & \text{if } \ast=-.
\end{array}\right.\kern-\nulldelimiterspace
\end{equation*}
Since $\beta$ is the preimage of a ray tail, the curve $L_2$ can be extended to contain an endpoint $c_2$, and if $c_2$ is not a critical point, then $f(c_2)$ must be the finite endpoint of $\gamma^n_{(\sigma(\s), \ast)}$. If $f(\{c_2\} \bm{\cdot} L_2)=\gamma^n_{(\sigma(\s), \ast)}$, then we define 
$$\gamma^{n+1}_{(\s, \ast)}\defeq \{c_2\} \bm{\cdot} L_2 $$
and the claim follows. This is because by construction, $\{c_2\} \bm{\cdot} L_2$ is either a right-extended or left-extended curve, depending only on whether $\ast=+$ or $\ast=-$, and by the inductive hypothesis, the same applies to the canonical tails $\gamma^{n}_{(\s, \ast)}$ and $\gamma^n_{(\sigma(\s), \ast)}$, and hence $\gamma^{n+1}_{(\s, \ast)}$ is a canonical tail. Otherwise, if $f(\{c_2\} \bm{\cdot} L_2)\neq \gamma^n_{(\sigma(\s), \ast)}$, the point $c_2$ must be a critical point, and we can define $L_3$ as the right or left extension of $L_2$ at $c_2$ following the same criterion as before. Iterating this process, we get a collection $\cdots\supset L_{i+1} \supset L_i \supset \cdots$ of right or left extended curves, all of the same type, contained in $\beta$. Since $\beta\subset J(f)$ and, by assumption, $J(f)\cap \AV(f)=\emptyset$, this process must converge. To see this, suppose that the piece of $\gamma^n_{(\sigma(\s), \ast)}$ from its finite endpoint $p$ to $f(c_1)$ is parametrized from $[0,1]$. Then, $f^{-1}(\gamma^n_{(\sigma(\s), \ast)}([0, 1]))\cap \beta$ is bounded, and so the sequence $\{L_{i}\}_{i\geq 1}$ converges to a canonical tail $L \subset \beta$ such that $f(L)=\gamma^n_{(\sigma(\s), \ast)}$. Consequently, by defining
$$\gamma^{n+1}_{(\s, \ast)}\defeq L, $$
\ref{item:tails_bijection} follows. The second part of the statement is a direct consequence of the construction process together with \eqref{eq_S0} in Definition \ref{def_initial_conf}.
\end{proof}	

We have shown in Theorem \ref{thm_signed} that the escaping set of any function satisfying its hypothesis, consists of a collection of $\Gamma$-curves, each of them being a union of nested canonical tails. However, we cannot assert that each of these canonical ray lands. This is however achieved in \cite{mio_splitting} for functions satisfying further assumptions. Next, we study the overlapping occurring within the collection of canonical rays:
\begin{prop}[Overlapping of $\Gamma$-curves] \label{prop_overlapping_Gamma} Following Theorem \ref{thm_signed}, for each $(\s, \ast) \in \Addr(f)_\pm$, either $\Gamma(\s, -)=\Gamma(\s, +)$ when $\Orb^{-}(\Crit(f)) \cap \Gamma(\s, \ast) =\emptyset$, or $\Gamma(\s, \ast)$ can be expressed as a concatenation 
\begin{equation}\label{def_concat2}
\Gamma(\s, \ast)=\cdots \bm{\cdot} \{c_{i+1}\}\bm{\cdot} \gamma^{i+1}_{i} \bm{\cdot} \{c_i\} \bm{\cdot} \cdots \bm{\cdot}\gamma^{1}_{0} \bm{\cdot} \{c_0\} \bm{\cdot} \gamma^{\infty}_{c_0},
\end{equation}
where $\{c_i\}_{i\in I}=\Orb^{-}(\Crit(f)) \cap \Gamma(\s, \ast)$, for each $i\geq 1$, if it exists, the curve $\gamma^{i+1}_{i}$ is a (bounded) piece of dynamic ray, and $\gamma^{\infty}_{c_0}$ is a piece of dynamic ray joining $c_0$ to infinity. In particular, in the latter case, the following properties hold for $\Gamma(\s, \ast)$:
\begin{enumerate}[label=(\Alph*)]
\item \label{itemA_Gamma} $\gamma^{\infty}_{c_0} \cup \{c_0\}= \Gamma(\s,- ) \cap \Gamma(\s, +)$ and $\gamma^{\infty}_{c_0}$ does not belong to any other $\Gamma$-curve.	
\item \label{itemB_Gamma} For each $i\geq 0$, the point $c_i$ belongs to exactly $2 \prod^{\infty}_{j=0}\deg(f,f^{j}(c_i))$ $\Gamma$-curves.
\item \label{itemC_Gamma} For each $i\geq 0$, $\gamma^{i+1}_i= \Gamma(\s, \ast)\cap \Gamma(\ultau, \star)$, where $\star\neq \ast$ and $\sigma^j(\ultau)=\sigma^j(\s)$ for some $j\geq 1$. Moreover, $\gamma^{i+1}_i$ does not belong to any other $\Gamma$-curve.
\end{enumerate}
\end{prop}

\begin{remark}
Note that $\mathcal{S}\subset J(f)$ by definition, and since any periodic critical point of $f$ is in $\C\setminus J(f)$, for each $i\in I$, the product $2 \prod^{\infty}_{j=0}\deg(f,f^{j}(c_i))$ in \ref{itemB_Gamma} is always finite. More specifically, let $n\in \N$ such that the point $c_i\in \gamma^n_{(\s,\ast)}\setminus \gamma^{n-1}_{(\s,\ast)}$. Then, by Theorem \ref{thm_signed} it holds that $f^n(c_i) \subset \gamma^0_{(\sigma^n(\s),\ast)}$, and as since by definition $\gamma^0_{(\sigma^n(\s),\ast)} \cap \Orb^{-}(\Crit(f))=\emptyset$, $ \prod^{\infty}_{j=0}\deg(f,f^{j}(c_i))=\prod^{n}_{j=0}\deg(f,f^{j}(c_i))<\infty$.
\end{remark}

\begin{proof}[Proof of Proposition \ref{prop_overlapping_Gamma}]
For a fixed $(\s, \ast) \in \Addr(f)_\pm$, the dichotomy and characterization of $\Gamma(\s, \ast)$ in the statement are a direct consequence of its definition in Theorem \ref{thm_signed}; more specifically, suppose that $\Orb^{-}(\Crit(f))\cap \Gamma(\s, \ast) =\emptyset$. Then, for each $n\geq 0$, there exists an inverse branch of $f^n$ defined in a neighbourhood of $\gamma^0_{(\sigma^n(\s),\ast)}$ that maps $\gamma^0_{(\sigma^n(\s),\ast)}$ bijectively to the connected component of $f^{-n}(\gamma^0_{(\sigma^n(\s),\ast)})$ that contains $\gamma^n_{(\s, -)}=\gamma^n_{(\s, +)}$. If $\Orb^{-}(\Crit(f))\cap \Gamma(\s, \ast) \neq\emptyset$, then since by Theorem \ref{thm_signed} the curve $\Gamma(\s, \ast)$ is a union of nested canonical tails, it must be of the form specified in \eqref{def_concat2}. 

For the rest of the proof, for each $n\in \N$ we refer to the elements in $$L(n)\defeq\left\{ \gamma^n_{(\ultau, \star)} \text{ : } (\ultau, \star)\in \Addr(f)_\pm \right\} $$
as \textit{curves of level $n$}. We shall use the following observation:
\begin{Claim}
Suppose that $z\in \gamma \in L(n)$ for some $n\in \N$. Then, $z\in \gamma^{m}_{(\ultau, \star)}$ for some $m>n$ and $(\ultau, \star)\in \Addr(f)_\pm$ if and only if $z\in \gamma^{n}_{(\ultau, \star)}$. 
\end{Claim}
In other words, if $z\in L(n)$ for some $n\in \N$, then all the $\Gamma$-curves $z$ belongs to, are determined by the curves of level $n$ it belongs to.
\begin{subproof}
If $z\in \gamma^{n}_{(\ultau, \star)}$ for some $(\ultau, \star)\in \Addr(f)_\pm$, then by Theorem \ref{thm_signed}, $z\in \gamma^{m}_{(\ultau, \star)}$ for all $m\geq n$. In order to prove the converse, suppose that $z\in \gamma^{m}_{(\ultau, \star)}$ for some $(\ultau, \star)$ and $ m>n$. Then, $f^n(z)\in \gamma^0_{(\sigma^n(\s), \ast)} \cup \gamma^{m-n}_{(\sigma^n(\ultau), \star)}$ by Theorem \ref{thm_signed}. However, since curves of level $0$ are pairwise disjoint and are contained only in curves of level $n$ that differ at most in the sign of their addresses, it must occur that $\sigma^n(\s)=\sigma^n(\ultau)$. This implies that $\gamma^0_{(\sigma^n(\ultau), \star)}=\gamma_{\sigma^n(\s)}^0$ and $z\in \gamma^{n}_{(\ultau, \star)}$.
\end{subproof}
For the rest of the proof, let us fix an arbitrary $(\s,\ast)\in \Addr(f)_\pm$. In order to prove \ref{itemA_Gamma}, we start recalling that all curves in a valid initial configuration are pairwise disjoint, and that by definition, $\gamma^0_{(\s, -)}=\gamma^0_\s=\gamma^0_{(\s, +)}$. Let $n\in \N$ be the smallest number such that $\gamma^n_{(\s, \ast)} \cap\Orb^{-}(\Crit(f)) \neq \emptyset$, and let $c_0$ be the point in that intersection of greatest potential in $\gamma^n_{(\s, \ast)}$. Then, there exists an inverse branch of $f^n$ defined in a neighbourhood of $f^n(\gamma^{\infty}_{c_0})\subset \gamma^0_{(\sigma^n(\s),\ast)}$ that maps $f^n(\gamma^{\infty}_{c_0})$ bijectively to $\gamma^{\infty}_{c_0}$, and thus, $\gamma^n_{(\s, -)}\cap \gamma^{\infty}_{c_0}=\gamma^n_{(\s, +)}\cap \gamma^{\infty}_{c_0}$. In particular, by the claim, $\gamma^{\infty}_{c_0}$ does not intersect any other canonical tails, and so \ref{itemA_Gamma} follows.
 
We prove items \ref{itemB_Gamma} and \ref{itemC_Gamma} simultaneously. Note that by the claim, all overlapping of $\gamma^m_{(\s, \ast)}$ with curves of level $n<m$ occur in $\gamma^n_{(\s, \ast)}\subset \gamma^m_{(\s, \ast)}$. By this and the definition of $\Gamma$-curves as a union of nested curves of level $m$, with $m\rightarrow \infty$, in order to prove \ref{itemB_Gamma} and \ref{itemC_Gamma}, it suffices to show that for any $n\geq 0$, \ref{itemB_Gamma} and \ref{itemC_Gamma} hold replacing in the statement of the proposition each $\Gamma(\s, \ast)$ by its restriction to $\gamma^n_{(\s, \ast)}$. We proceed to do so by induction on $n$. For $n=0$, since curves of level $0$ do not contain (preimages of) critical points, the statements hold trivially. Suppose that \ref{itemB_Gamma} and \ref{itemC_Gamma} hold for some $n\in \N$, and we shall see they hold for $n+1$.

Let us consider the curve $\gamma^{n+1}_{(\s, \ast)}$. By Theorem \ref{thm_signed}, $f(\gamma^{n+1}_{(\s, \ast)})=\gamma^{n}_{(\sigma(\s), \ast)}$ and by the inductive hypothesis, the statements hold for both $\gamma^{n}_{(\s, \ast)}$ and $\gamma^{n}_{(\sigma(\s), \ast)}$. In particular, since $\gamma^{n}_{(\s, \ast)}\subseteq \gamma^{n+1}_{(\s, \ast)}$, for all $\gamma^{i+1}_i$ and $c_i$ contained in $\gamma^{n}_{(\s, \ast)}$ for some $i\in I$, by the claim, \ref{itemB_Gamma} and \ref{itemC_Gamma} hold. Thus, if $\gamma^{n+1}_{(\s, \ast)}= \gamma^{n}_{(\s, \ast)}$, we are done. Otherwise, it suffices to prove that they hold for all 
$$c_i \in \beta\defeq \gamma^{n+1}_{(\s, \ast)}\setminus \gamma^{n}_{(\s, \ast)} \quad \text{ and } \quad\gamma^{i+1}_i \cap \beta \neq \emptyset \quad \text{ for some } i\in I.$$
If $\beta \cap \Crit(f)=\emptyset$, then, arguing as before, there exists a neighbourhood of $\beta$ that maps injectively to $\gamma^{n}_{(\sigma(\s), \ast)}$. In particular, for each curve of level $n$ that contains $f(\beta)$, there exists a unique curve of level $n+1$ that maps to it and also contains $\beta$. Then, by the inductive hypothesis applied to curves of level $n$, \ref{itemB_Gamma} and \ref{itemC_Gamma} hold for $\beta \cup \gamma^{n}_{(\s, \ast)}= \gamma^{n+1}_{(\s, \ast)}$. 

Otherwise, let $c \in \beta \cap \Crit(f)$ be the critical point of maximal potential in $\beta$. Then, by definition, the map $f$ acts like $z \mapsto z^{\deg(f,c)}$ locally around $c$. By the inductive hypothesis, $f(c)$ belongs to $N\defeq 2 \prod^{\infty}_{j=1}\deg(f,f^{j}(c))$ curves of level $n$ that by \ref{itemC_Gamma}, overlap pairwise. Let $\mathcal{L}(c)$ be the set of curves in $\beta \setminus \Crit(f)$ for which $c$ is an endpoint. Then, the cardinal of $\mathcal{L}(c)$ is either $\deg(f,c)\cdot N$, or $2\deg(f,c)\cdot N$, depending on whether $f(c)$ is or not the endpoint of $\gamma^{n}_{(\sigma(\s), \ast)}$. We will assume without loss of generality that the second case occurs, since the argument in the first one is a simplified version of the one to follow. Let us subdivide the curves in $\mathcal{L}(c)$ into the respective subsets $\mathcal{L}_b$ and $\mathcal{L}_u$ of curves that are mapped to the bounded or unbounded component of $\gamma^n_{(\sigma(\s), \ast)}\setminus \{f(c)\}$.

In particular, since $c\in \beta$, $\gamma^n_{(\s, \ast)}$ must belong to a curve in $\mathcal{L}_u$. Moreover, by Theorem \ref{thm_signed} and the inductive hypothesis, each curve in $\mathcal{L}_u$ contains a pair of curves $\gamma^n_{(\ultau, +)}$ and $\gamma^n_{(\ul{\alpha}, -)}$ for some $\ultau, \ul{\alpha}\in \Addr(f)$ such that $\sigma(\ultau)=\s=\sigma(\ul{\alpha})$. Since $f$ maps each curve in $\mathcal{L}_u$ injectively to $\gamma^n_{(\sigma(\s), \ast)}\setminus \{f(c)\}$, arguing as before, each curve in $\mathcal{L}_u$ belongs to two curves of level $n+1$ that extend the curves of level $n$ they contain. In addition, since curves of level $n+1$ are canonical tails, a curve from $\mathcal{L}_b$ is a bristle for each of them. In particular, following the extending criterion from \ref{dis_canonical}, each curve in $\mathcal{L}_b$ is both a left bristle for a left-extended curve in $\mathcal{L}_u$, and a right bristle for a right-extended curve in $\mathcal{L}_u$. In particular, one of these right and left-extended curves must belong to $\gamma^{n+1}_{(\s, \ast)}$, and we have shown that items \ref{itemB_Gamma} and \ref{itemC_Gamma} hold for the restriction of $\gamma^{n+1}_{(\s, \ast)}$ to the unbounded component of $\gamma^{n+1}_{(\s, \ast)}\setminus (\Crit(f)\setminus \{c\})$. If we denote the bounded component by $\delta$, repeating this process iteratively for each critical point in $\delta$, since $\delta$ is bounded, the process must converge and the statements follow.
\end{proof}

\begin{defn}[Signed addresses for escaping points]\label{def_addrpm} Under the assumptions of Theorem~\ref{thm_signed}, for each $z\in \mathcal{S} \supset I(f)$, we say that $z$ has \emph{signed (external) address} $(\s, \ast)$ if $ z\in \Gamma(\s, \ast)$, and we denote by $\Addr(z)_\pm$ the set of all signed addresses of $z$.
\end{defn}

\begin{observation}[Escaping points have at least two signed addresses]\label{obs_formula_addr} By Proposition~\ref{prop_overlapping_Gamma}, for each $z\in \mathcal{S}$,
\begin{equation}\label{eq_singed_adddr}
\# \Addr(z)_\pm= 2 \prod^{\infty}_{j=0}\deg(f,f^{j}(z))<\infty.
\end{equation}
Therefore, each point in $\mathcal{S}\supset I(f)$ has at least two signed addresses.
\end{observation}

\begin{observation}[Universality of canonical tails]\label{obs_universalcanonical} We note that the concept of \textit{canonical tail} for a criniferous function $f\in \B$ is defined in \ref{dis_canonical} independently of the choice of fundamental domains, and thus external addresses, for $f$. However, for each choice of $\Addr(f)_\pm$, since any canonical tail, or more generally any ray tail, escapes uniformly to infinity, it must contain some Julia constituent. Then, it follows from Theorem \ref{thm_signed} and Proposition \ref{prop_overlapping_Gamma} that if $\gamma \subset I(f)$ is a canonical tail, then there exists $(\s, \ast) \in \Addr(f)_\pm$ such that $\gamma=\gamma^n_{(\s, \ast)}$ for some $n\geq 0$.
\end{observation}

\begin{observation}[Landing of canonical rays implies landing of all rays]\label{rem_mother}
For a criniferous function $f \in \B$ such that $J(f)\cap \AV(f)=\emptyset$, showing that all its canonical rays land suffices to conclude that all its dynamic rays land. This is because we have shown in Theorem \ref{thm_signed} that $I(f)\subset \mathcal{S}= \bigcup_{(\s, \ast) \in \Addr(f)_\pm} \Gamma(\s, \ast)$, and so, any dynamic ray $\gamma$ must belong to $\mathcal{S}$. In particular, if a ray $\gamma$ is canonical, then by Observation \ref{obs_universalcanonical}, it must be contained in $\Gamma(\s, \ast)$ for some signed address $(\s, \ast)$, and if $\gamma$ is not canonical, then $\gamma$ must be a concatenation at (preimages of) critical points of pieces of ray tails, where instead of extending as in Definition \ref{def_leftrightextensions}, different choices of bristles are made.
\end{observation}

\begin{observation}[Equivalence of orders]\label{obs_same_order}
We note that locally, the anticlockwise order of radial segments used in Definition \ref{def_leftrightextensions} agrees with the order in \eqref{eq_orderAfpm} for the addresses in $\Addr(f)_\pm$, since in \eqref{eq_orderinfty}, we chose positive orientation. More precisely, we aim to show in \S \ref{sec_hands} that by construction, given a converging sequence of points $\{z_k\}_k \subset \mathcal{S}$, for each $k>0$, there is $(\ul{s}_k, \ast_k) \in \Addr(z_k)_\pm$ such that 
\begin{equation}\label{eq_convernewaddr}
\text{if} \quad z_k\rightarrow z , \quad \text{ then } \quad (\s_k, \ast_k)\rightarrow (\s, \ast) \quad \text {as } \quad k\rightarrow \infty
\end{equation}
for some $(\s, \ast)\in \Addr(z)_\pm$. Compare to \eqref{eq_sktos}. 
\end{observation}
\section{Fundamental hands and inverse branches}\label{sec_hands}
The standing assumptions for any $f$ in this section are the following: 
\begin{itemize}
\item $f\in \B$ is criniferous with $J(f)\cap \AV(f)=\emptyset$,
\item $P(f)\setminus I(f)$ is bounded and the set $S_I\defeq S(f)\cap I(f)$ is finite.
\end{itemize}

Recall that by Theorem~\ref{thm_signed}, we can define for each $(\s,\ast) \in \Addr(f)_\pm$ and $n\geq 0$, a canonical tail $\gamma^n_{(\s, \ast)}$ such that $f^n$ maps $\gamma^n_{(\s, \ast)}$ bijectively to $\gamma^0_{(\sigma^n(\s), \ast)}$. Using this, in this section, we define for each $(\s,\ast) \in \Addr(f)_\pm$ and $n\geq 0$ an inverse branch of $f^n$ in a neighbourhood $U$ of $\gamma^0_{(\sigma^n(\s), \ast)}$ such that the image of this inverse branch contains $\gamma_{(\s, \ast)}^n$. In addition, $U$ will be chosen so that there exists an interval of signed addresses containing $(\s, \ast)$ such that the same inverse branch with analogous properties can be taken for all addresses in the interval; see Theorem \ref{lem_inverse_hand}. We are able to achieve this result thanks to the consistency on taking always either right or left extensions in the definition of canonical tails, together with the equivalence of orders pointed out in Observation \ref{obs_same_order}. In order to define these inverse branches, we introduce the concept of \textit{fundamental hands}, that in a rough sense are $n$-th preimages of certain simply connected subsets of fundamental domains, where $f^n$ is injective, see Definition \ref{def_hands}.

\begin{remark}
We suggest the reader familiar with \cite{lasse_dreadlocks}, to compare the notion of fundamental hands with that of \textit{fundamental tails} for postsingularly bounded functions, introduced in \cite[Section 3]{lasse_dreadlocks} in order to define \textit{dreadlocks}. Since for the functions we consider in this section, we face the additional challenge of the presence of critical points in their escaping sets, the existence of a sensible generalization of the concept of dreadlocks for these functions is a priori not obvious to us. See \cite[\S 5.2]{mio_thesis} for further discussion on this topic.
\end{remark}

We start by finding a convenient choice of parameters that determine an alphabet of fundamental domains, with the aim of simplifying arguments in future proofs. Recall that by Theorem \ref{thm_signed}, every point in $S_I$ is the endpoint of at least one canonical tail. 

\begin{prop}[Parameters to define fundamental domains]\label{prop_PF_fundamental}
Let $f\in \B$ satisfying the standing hypotheses. For each $z \in S_I$, let $\gamma_z$ be a canonical tail with finite endpoint $z$. Then, there exists a Jordan domain $D\supset S(f) \cup (P(f)\setminus I(f))$ and an arc $\delta \subset \C \setminus\overline{f^{-1} (\C\setminus D)}$ connecting a point of $\overline{D}$ to infinity such that
\begin{equation}\label{eq_W0}
W_{-1} \defeq \C\setminus( \overline{D}\cup\delta)
\end{equation}
has the following property: if for some $n\geq 1$, a connected component $\tau$ of $f^{-n}(W_{-1})$ contains a point $z \in S_I$, then $\gamma_z \subset \tau$. Moreover, for any such $\tau$, $(P(f)\setminus I(f))\cap \tau=\emptyset$.
\end{prop}

\begin{proof}
By the assumptions on $f$, we can choose a disk $\D_{R_0}$ for some $R_0>0$ sufficiently large so that $S(f) \cup (P(f)\setminus I(f)) \subset \D_{R_0}$. Let $\mathcal{T}_{\D_{R_0}}\defeq f^{-1}(\C \setminus \D_{R_0})$ be the corresponding set of tracts. Since each ray tail escapes to infinity uniformly, see Definition \ref{def_ray}, for each $z \in S_I$, there exists a natural number $N(z)>0$ such that $f^n(\gamma_z)\subset \mathcal{T}_{\D_{R_0}}$ for all $n\geq N(z)$. Let us consider the set of ray tails
$$\mathcal{R} \defeq \{f^n(\gamma_z): z\in S_I \text{ and } 0 \leq n\leq N(z)\},$$
which has finitely many elements as, by assumption, $\#S_I$ is finite. Note that if $\gamma$ is a ray tail, then by definition, $\lim_{t \to \infty}f(\gamma(t))=\infty$, and in particular, there exists a constant $R(\gamma)>0$ such that $\gamma\setminus \D_{R(\gamma)} \subset \mathcal{T}_{\D_{R_0}}$. Since $\#\mathcal{R}<\infty$, there exists a finite constant
$$R_1\defeq \max_{\gamma \subset \mathcal{R}} R(\gamma).$$
Thus, for all $\gamma \in \mathcal{R}$, $\gamma\setminus \D_{R_1} \subset \mathcal{T}_{\D_{R_0}}$.
Let us define tracts for $f$ with respect to $\D_{R_1}$, that is, $ \T_{\D_{R_1}}\defeq f^{-1}(\C \setminus \D_{R_1}).$ We can assume without loss of generality that $\D_{R_0}\subset \D_{R_1}$, since otherwise we can replace $R_1$ by a bigger constant. Thus, it holds that $\T_{\D_{R_1}} \subset \T_{\D_{R_0}}$, and by construction, for all $n\geq 0$ and $z\in S_I$, $f^n(\gamma_z)\subset (\D_{R_1} \cup \T_{\D_{R_0}})$. Consequently, we can choose a curve $\tilde{\delta}\subset \C \setminus (\overline{\mathcal{T}}_{\!\D_{R_0}}\cup \D_{R_1})$ connecting a point in $\partial \D_{R_1}$ to infinity. In particular,
\begin{equation}\label{eq_specialdef_fund}
f^n(\gamma_z) \cap \tilde{\delta} =\emptyset \quad \text{ for all } \quad z\in S_I \quad \text{ and } \quad n\geq 0. 
\end{equation}
Note that \eqref{eq_specialdef_fund} is equivalent to 
\begin{equation*}
\bigcup_{z\in S_I} \gamma_z \;\cap \; \bigcup_{n\geq 0}f^{-n}(\tilde{\delta})= \emptyset.
\end{equation*}
However, if we defined a set $W_{-1}$ as $\C\setminus(\overline{\D}_{R_1}\cup \tilde{\delta})$, then $W_{-1}$ would not satisfy the property on connected components of its preimages specified in the statement, since the curves $\{\gamma_z\}_{z\in S_I}$ could a priori intersect some preimage of $\partial \D_{R_1}$. Note that such property is equivalent to saying that, for some appropriate $\overline{D}$ and $\delta$, if $f^n(z) \in \C\setminus (\overline{D}\cup \delta)$ for some $z \in S_I$ and $n\geq 1$, then $f^{n}(\gamma_z) \subset \C\setminus( \overline{D}\cup\delta).$ Hence, in order to define $W_{-1}$ satisfying the property we are looking for, by \eqref{eq_specialdef_fund}, it suffices to find a domain $D\supset \D_{R_1}$ and a curve $\delta \subset \tilde{\delta}$ such that $\overline{D}\cap \delta$ is a single point, and so that 
\begin{equation}\label{eq_proofW0}
\text{if } \quad z \in S_I \text{ and } f^n(z) \in (\C \setminus \overline{D}) \text{ for some }n\in \N, \quad \text{ then } \quad f^{n}(\gamma_z) \subset (\C \setminus \overline{D}). 
\end{equation}
Thus, our next aim is to find a domain $D$ for which \eqref{eq_proofW0} holds.

Arguing as before, by definition of ray tail, for each $z \in S_I$, there exists a constant $M(z)\in \N$ such that $f^{m}(\gamma_z)\subset(\C \setminus \overline{\D}_{R_1})$ for all $m\geq M(z)$. Hence, there exists a constant $Q\geq R_1$ such that 
$$\D_Q \supset \left\{f^n(z) : \ z\in S_I \ \text{ and } \ 0 \leq n\leq M(z) 
\right\}\eqdef \mathcal{P},$$
where the set $\mathcal{P}$ has only finitely many elements. As before, by definition of ray tail, only finitely many rays that are of the form $f^m(\gamma_z)$ for some $m\geq M(z)$ and $z \in S_I$ intersect $\D_Q$. Hence, we can find a domain $D$ such that $\D_{R_1} \cup \mathcal {P} \subset D \subset \D_Q$, and so that $D \cap f^m(\gamma_z) =\emptyset$ for all $m\geq M(z)$ and $z \in S_I$. This means that since $\mathcal{P}\subset D$, the hypothesis in \eqref{eq_proofW0} can only hold for $f^m(z)$ with $m\geq M(z)$, but by construction and since $\T_D\subset T_{\D_{R_1}}$, the thesis in \eqref{eq_proofW0} always holds for these cases. Thus, \eqref{eq_proofW0} is always true for our choice of $D$. Defining $\delta$ as the unbounded connected component of $\tilde{\delta}\setminus D$, the proof of the first part of the statement is concluded. The fact that for any connected component $\tau$ of $f^{-n}(W_{-1})$ it holds that $(P(f)\setminus I(f))\cap \tau=\emptyset$, is a consequence of $(P(f)\setminus I(f))\subset \D_{R_1}\subset D$ and $(P(f)\setminus I(f))$ being forward-invariant.
\end{proof}

\noindent We are now ready to define the basic objects of this section.
\begin{defn}[Fundamental hands]\label{def_hands} Following Proposition \ref{prop_PF_fundamental}, let $W_{-1}$ be the domain from \eqref{eq_W0}. Then, for each $n\geq 0$, define inductively $$ X_n\defeq \bigcup_{z\in W_{n-1} \cap S_I} \gamma_z \qquad \text{ and } \qquad W_n\defeq f^{-1}\left(W_{n-1} \setminus X_n\right).$$
For every $n\geq 0$, each connected component of $W_n$ is called a \textit{fundamental hand} of level $n$.
\end{defn}

That is, in a rough sense, we take successive preimages of $W_{-1}$, removing at each step some curves in $\{\gamma_z\}_{z \in S_I}$ whenever a point in $S_I$ belongs to some component of a preimage. In particular, $X_0=\emptyset$ and fundamental hands of level $0$ are fundamental domains for $f$. The choice of $W_{-1}$ in Proposition \ref{prop_PF_fundamental} has been made so that the following basic properties of fundamental hands hold:
\begin{prop}[Facts about fundamental hands]\label{prop_basic_hand} Fundamental hands are unbounded, simply connected, and any two of the same level are pairwise disjoint. Moreover, each fundamental hand of level $n>1$ is mapped univalently under $f$ to a fundamental hand of level $n-1$.
\end{prop}
\begin{proof}
We prove all facts simultaneously using induction on $n$. For $n=0$, fundamental hands are fundamental domains, and so the statement follows trivially. Let us assume that the statement holds for some $n-1 \in \N$, and we shall see that it holds for $n$. Let $\tau$ be a fundamental hand of level $n$. Then, by definition, its image $f(\tau)$ is contained in a fundamental hand $\tilde{\tau}$ of level $n-1$, and 
\begin{equation}\label{eq_partialtau}
\partial \tilde{\tau}\subset \Big(f^{-n}(\partial W_{-1}) \cup \bigcup_{\substack{0<i<n\\	z\in S_I}} f^{-i}(\gamma_z)\Big).
\end{equation}
By Proposition \ref{prop_PF_fundamental}, $X_n$ does not intersect $f^{-n}(\partial W_{-1})$. 
Since by \eqref{eq_partialtau}, all other connected components that might form $\partial \tilde{\tau}$ are preimages of ray tails, and thus in $I(f)$, $X_n \cap (\partial \tilde{\tau}\setminus f^{-n}(\partial W_{-1}))$ must be simply connected: Otherwise, there would be a domain enclosed by pieces of ray tails that escapes uniformly to infinity, contradicting tha$\mathcal{B}$, \cite{eremenkoclassB}. Thus, by this and using the inductive hypothesis, $\tilde{\tau} \setminus X_n$ is an unbounded, simply connected domain. Since $f$ is an open map, the same holds for $\tau$.
 
In order to see that $f\vert_{\tau}$ is injective, note that all singular values contained in $W_{n-1}$ also belong to $X_n$, and hence $\tau \cap \Crit(f)=\emptyset$. This implies that the restriction $f\vert_{\tau}$ is a covering map, and all inverse branches of $f$ in the domain $f(\tau)$ can be continued. Moreover, as we have seen in the previous paragraph, $f(\tau)=\tilde{\tau}\setminus X_n$ is simply connected. Thus, all arcs in $f(\tau)$ with fixed endpoints are homotopic, and hence, by the Monodromy Theorem (see \cite[Theorem 2, p.295]{Ahlfors}), given any two homotopic curves in $f(\tau)$, for an inverse branch of $f$ defined in a neighbourhood of their starting endpoint, all its analytic continuations along the curves lead to the same values at the terminal endpoint, and so $f\vert_{\tau}$ is injective. By the inductive hypothesis, fundamental hands of level $n-1$ are pairwise disjoint, and since fundamental hands of level $n$ are the connected components of the preimages of subsets of those hands, they are also pairwise disjoint.
\end{proof}


Next, we define external addresses for $f$ using the alphabet of fundamental domains $\mathcal{A}(D,\delta)$, where $D$ and $\delta$ are provided by Proposition \ref{prop_PF_fundamental}. 
As usual, we denote by $\Addr(f)$ the set of admissible external addresses for $f$.

In the next proposition, that will serve us as an auxiliary result to prove Proposition~\ref{prop_hands_fortails}, we show that fundamental hands of any level always intersect at least one fundamental domain, and that this intersection is unbounded. As a consequence, we obtain that fundamental hands contain Julia constituents.
\begin{prop}[Fundamental hands contain Julia constituents]\label{prop_hands_addr} Let $\tau$ be a fundamental hand of level $n$ for some $n\geq 0$. Then, there exists at least one fundamental domain $F_1$ such that $\tau \cap F_1$ is unbounded. Moreover, if $J_{\underline{\omega}} \subset \tau \setminus X_n$ for some $\underline{\omega}\in \Addr(f)$, then for each fundamental domain $F_0$, there exists a unique fundamental hand $\tilde{\tau}$ of level $n+1$ such that $f(\tilde{\tau})\subset \tau$ and $J_{F_0\underline{\omega}} \subset F_0 \cap\tilde{\tau}$. In particular, there is $\s \in \Addr(f)$ with $J_\s\subset \tau$. 
\end{prop}
\begin{proof}
In order to prove the first claim, we proceed by induction on the level $n$ of the hand ~$\tau$. For $n = 0$, $\tau$ is a fundamental domain and so the first claim is trivial. Suppose that it is true for some $n-1 \in \N$. In order to see that it also holds for $n$, note that by definition of fundamental hands, $f(\tau)\subset \tau_2$, where $\tau_2$ is a fundamental hand of level $n-1$. Then, by the inductive hypothesis, there exists at least one fundamental domain $F_1\in \mathcal{A}(D, \gamma)$ such that $F_1\cap \tau_2$ is unbounded. Let $U$ be the unbounded connected component of $\unbdd{F}_1\cap (f(\tau) \setminus \overline{D})$. Then, since by Proposition \ref{prop_basic_hand}, $f\vert_{\tau}$ maps to $f(\tau)$ bijectively and since $U \subset W_{-1}$, there exists a unique fundamental domain $F_0$ containing the unbounded set $ f^{-1}(U) \cap \tau$, and so we have proved the first claim.  
	
For second claim, let $\underline{\omega}=F_1F_2\ldots\in \Addr(f)$ be as in the statement. Thus, $J_{\underline{\omega}} \subset \unbdd{F}_1$. Then, for any other fundamental domain $F_0$, by the same argument as before, $f^{-1}(\unbdd{F}_1)\cap F_0$ is an unbounded set that by definition contains $J_{F_0\underline{\omega}}$. 
In particular, by definition of fundamental hands and since they are pairwise disjoint, there is a unique fundamental hand $\tilde{\tau} \subset f^{-1}(\tau \setminus X_n)$ of level $n+1$ that contains $J_{F_0\underline{\omega}}$, as we wanted to show.

Finally, we construct $\s \in \Addr(f)$ such that $J_\s\subset \tau\eqdef \tau_0$. Note that for each $0\leq j\leq n$, $f^j(\tau)\subset \tau_j$ for some fundamental hand $\tau_j$ of level $n-j$. In particular, by the first part of the proposition, for each fundamental hand $\tau_j$, there exists a fundamental domain $F_j$ such that $\tau_j\cap F_j$ is unbounded. Recall that the curves $\{\gamma_z\}_{z\in S_I}$ that form the sets $X_j$ are canonical tails. Then, by Theorem \ref{thm_signed} and Proposition \ref{prop_overlapping_Gamma} together with Observation \ref{obs_universalcanonical}, each $\gamma_z$ contains $J^\infty_{\ul{\alpha}}$ for exactly one element $\ul{\alpha}\in \Addr(f)$, and consequently, by \eqref{eq_inclJs}, the same holds for each of the canonical tails in
$$\mathcal{R} \defeq \{f^j(\gamma_z): z\in S_I \text{ and } 0 \leq j\leq n\}.$$
The set $\mathcal{R}$ has finite cardinality, and thus, so does the set 
$$\Addr(\mathcal{R})\defeq \{\s \in \Addr(f) :J^\infty_\s \subset \gamma \text{ for some } \gamma\in \mathcal{R} \}.$$
Since $\Addr(f)$ is an uncountable set, we can choose any bounded $\underline{\alpha}\in \Addr(f)$ such that $\s\defeq F_0F_1\ldots F_{n}\underline{\alpha} \notin \Addr(\mathcal{R})$. Then, since $\s$ is also a bounded address, by Theorem \ref{thm_props_Js}, $J_\s \neq \emptyset$, and so $\s \in \Addr(f)$. By construction, $J^\infty_\s\cap X_j=\emptyset$ for all $0\leq j\leq n$, and in particular, $J_{F_n\underline{\alpha}} \subset \tau_n \setminus X_n$. Then, by the second part of the proposition, $J_{F_{n-1}F_n\underline{\alpha}} \subset \tau_{n-1}\cap F_{n-1}$. Iterating this argument $n-2$ further times, we see that $J_\s \subset \tau$.
\end{proof}

As discussed at the beginning of the section, we are interested in finding neighbourhoods of canonical tails on which inverse branches are well-defined. These neighbourhoods will be provided by images of closures of fundamental hands. 

Let $\Addr(f)_\pm$ the space of signed addresses, defined using our choice of $\Addr(f)$. Moreover, let $$\mathcal{C}\defeq \left\{\gamma^n_{(\s, \ast)} : n\geq 0 \text{ and }(\s, \ast)\in \Addr(f)_\pm\right\}$$ be a set of canonical tails provided by Theorem \ref{thm_signed} for any valid initial configuration. In particular, one can always use the configuration $\{J^\infty_\s\}_{\s \in \Addr(f)}$, see Observation \ref{obs_config_crini}.

Then, in the next proposition we show that each canonical tail $\gamma^n_{(\s, \ast)} \in \mathcal{C}$ belongs to the closure of at least one and at most two fundamental hands of level $n$. In addition, we assign to each canonical tail a fundamental hand that will allow us to define the desired inverse branches in the following Theorem \ref{lem_inverse_hand}. Recall from Proposition \ref{prop_hands_addr} that given any fundamental hand, we can find a Julia constituent contained in it.

\begin{prop}[Fundamental hands for canonical tails] \label{prop_hands_fortails} 
For each $\gamma^n_{(\s, \ast)}\in \mathcal{C}$, exactly one of the following holds: 
\begin{enumerate}[label=(\Alph*)]
\item \label{item_unique_hand}There exists a unique fundamental hand $\tau$ of level $n$ such that $\gamma^n_{(\s, \ast)}\subset \overline{\tau}.$ We denote
$$\tau_n(\s, \ast)\defeq \tau \cup \gamma^n_{(\s, \ast)}.$$
\item \label{item_two_hands} The curve $\gamma^n_{(\s, \ast)}$ belongs to the boundary of exactly two fundamental hands $\tau$ and $\tilde{\tau}$ of level $n$. Let $\ul{\upsilon}, \ul{\omega} \in \Addr(f)$ such that $J^\infty_{\ul{\upsilon}}\subset \tau$ and $J^\infty_{\ul{\omega}} \subset \tilde{\tau}$. Then, we denote
\begin{equation}\label{eq_deftau_n}
\tau_n(\s, \ast)\defeq \renewcommand{\arraystretch}{1.5}\left\{\begin{array}{@{}l@{\quad}l@{}}
\tau \cup \gamma^n_{(\s, \ast)} & \text{if } [\ul{\upsilon},\ul{s},\ul{\omega}]_\ell \text{ and } \ast=- \ \text{ or } \ [\ul{\omega},\ul{s},\ul{\upsilon}]_\ell \text{ and } \ast=+; \\
\tilde{\tau} \cup \gamma^n_{(\s, \ast)} & \text{otherwise,}
\end{array}\right.\kern-\nulldelimiterspace
\end{equation}
where ``$[ \cdot ]_\ell$'' is the cyclic order in addresses defined in \eqref{eq_cyclicell}. 
\end{enumerate}	
\end{prop}

\begin{remark}
The definition of $\tau_n(\s, \ast)$ in case \ref{item_two_hands} is independent of the choice of addresses $\ul{\upsilon}, \ul{\omega} \in \Addr(f)$ such that $J^\infty_{\ul{\upsilon}}\subset \tau$ and $J^\infty_{\ul{\omega}} \subset \tilde{\tau}$. To see this, note that by definition of fundamental hands, the connected component $T$ of $\partial \tau$ that contains $J^\infty_{\s}$ separates the plane, and in particular, $\tau$ and $\tilde{\tau}$ lie in different components of $\C \setminus T$. Moreover, $T$ must contain a (preimage of a) critical point, and so $J^\infty_{\ul{\alpha}} \subset T$ for some other $\ul{\alpha} \in \Addr(f)$. Then, we can define a linear order ``$<$'' in $\Addr(f)$ by cutting $\ul{\alpha}$. In particular, either $\ul{\upsilon}<\s<\ul{\omega} $ or $\ul{\upsilon}>\s>\ul{\omega} $ for all pairs of addresses $\ul{\upsilon}$ and $\ul{\omega}$ so that $J^\infty_{\ul{\upsilon}} \subset \tau$ and $J^\infty_{\ul{\omega}} \subset \tilde{\tau}$, and using the equivalence \eqref{eq_cyclicorder}, the claim follows.
\end{remark}
\begin{proof}[Proof of Proposition \ref{prop_hands_fortails}] 
Firstly, recall that by Observation \ref{rem_mother}, each canonical tail belongs to a curve $\gamma^m_{(\ul{\alpha}, \star)}$ for some $m\geq 0$ and $(\ul{\alpha}, \star)\in \Addr(f)_\pm$. In particular, this is the case for the canonical tails $\{\gamma_z\}_{z\in S_I}$ we chose in the definition of fundamental hands. We remark that the curves in $\{\gamma_z\}_{z\in S_I}$ might not be pairwise disjoint, but since $f\in \B$, $I(f)$ has empty interior \cite{eremenkoclassB}, and thus, each connected component of $\C \setminus \bigcup_{z\in S_I} \gamma_z$ is simply connected. Then, the boundaries of fundamental hands are either connected components of preimages of curves in $\{\gamma_z\}_{z\in S_I}$, or preimages of the boundary of the set $W_{-1}$ from Proposition \ref{prop_PF_fundamental}. In particular, by Propositions \ref{prop_overlapping_Gamma} and \ref{prop_PF_fundamental}, each curve $\gamma^n_{(\s, \ast)}$ might only intersect boundaries of fundamental hands at (preimages of) critical points, or might share with them a segment between any two of those preimages.

Let us fix a curve $\gamma^n_{(\s, \ast)}\in \mathcal{C}$. Since, by Theorem \ref{thm_signed} and Proposition \ref{prop_overlapping_Gamma} together with Observation \ref{obs_universalcanonical}, each curve in $\{\gamma_z\}_{z \in S_I}$ contains a curve $\gamma^0_{\ul{\alpha}}$ for exactly one address $\ul{\alpha}\in \Addr(f)$, and the same holds for the unbounded components of $f^{-n}(\gamma_z)\setminus \Crit(f)$ for all $n\in \N$, the curve $\gamma_{(\s, \ast)}^0$ is either totally contained in a hand of level $n$, or belongs to the boundary of two such hands whenever $f^i(\gamma_{(\s, \ast)}^0) \subset X_i$ for some $1\leq i \leq n$. We subdivide the proof into these two cases:
\begin{itemize}[noitemsep,wide=0pt, leftmargin=\dimexpr\labelwidth + 2\labelsep\relax]
\item Suppose that $\gamma_{(\s, \ast)}^0 \subset \tau$, where $\tau$ is a hand of level $n$. If $\gamma_{(\s, \ast)}^n \subset \tau$, then case \ref{item_unique_hand} holds. Otherwise, let $x\in \gamma_{(\s, \ast)}^n$ be the point of greatest potential that also belongs to $\partial\tau $. By Propositions \ref{prop_PF_fundamental} and \ref{prop_overlapping_Gamma}, $x$ must be a (preimage of a) critical point, and there must be at least two canonical tails in $\partial \tau$ that also contain $x$. Recall that bristles are defined for canonical tails as either the successor or predecessor (bounded) segment in the circular order of segments around $x$; see Figure \ref{fig:branches}. Hence, the bristles of the unbounded component of $\gamma^n_{(\s, \ast)}\setminus \{x\}$ must lie between this curve and two unbounded components of $\partial\tau\setminus \{x\}$ that contain $x$ as an endpoint. Thus, these bristles belong to either $\partial \tau$ or $\tau$. If the bounded component of $\gamma^n_{(\s, \ast)}\setminus \{x\}$ does not contain any other (preimage of a) critical point, then we are done. Otherwise, $\gamma^n_{(\s, \ast)}$ might intersect another boundary component of $\partial \tau$ in a point $y$, and by the same argument, the bristles of the corresponding unbounded component of $\gamma^n_{(\s, \ast)}\setminus \{y\}$ must again lie in $\overline{\tau}$. Thus, case \ref{item_unique_hand} holds for $\gamma^n_{(\s, \ast)}$.

\begin{figure}[h]
\centering
\begingroup%
  \makeatletter%
  \providecommand\color[2][]{%
    \errmessage{(Inkscape) Color is used for the text in Inkscape, but the package 'color.sty' is not loaded}%
    \renewcommand\color[2][]{}%
  }%
  \providecommand\transparent[1]{%
    \errmessage{(Inkscape) Transparency is used (non-zero) for the text in Inkscape, but the package 'transparent.sty' is not loaded}%
    \renewcommand\transparent[1]{}%
  }%
  \providecommand\rotatebox[2]{#2}%
  \newcommand*\fsize{\dimexpr\f@size pt\relax}%
  \newcommand*\lineheight[1]{\fontsize{\fsize}{#1\fsize}\selectfont}%
  \ifx\svgwidth\undefined%
    \setlength{\unitlength}{314.64566929bp}%
    \ifx\svgscale\undefined%
      \relax%
    \else%
      \setlength{\unitlength}{\unitlength * \real{\svgscale}}%
    \fi%
  \else%
    \setlength{\unitlength}{\svgwidth}%
  \fi%
  \global\let\svgwidth\undefined%
  \global\let\svgscale\undefined%
  \makeatother%
  \begin{picture}(1,0.62162162)%
    \lineheight{1}%
    \setlength\tabcolsep{0pt}%
    \put(0,0){\includegraphics[width=\unitlength,page=1]{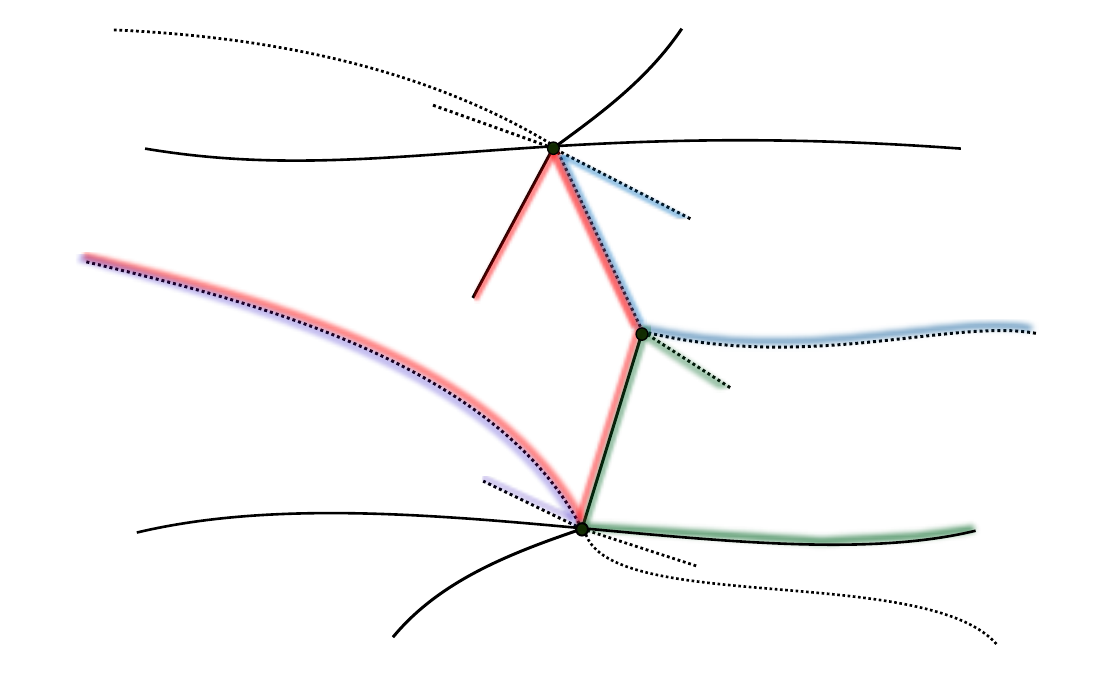}}%
    \put(0.13404596,0.39878471){\color[rgb]{0.83137255,0,0}\makebox(0,0)[lt]{\lineheight{1.25}\smash{\begin{tabular}[t]{l}\fontsize{9pt}{1em}$\gamma^n_{(\underline{s}, -)}$\end{tabular}}}}%
    \put(0.7867219,0.44512477){\color[rgb]{0,0,0}\makebox(0,0)[lt]{\lineheight{1.25}\smash{\begin{tabular}[t]{l}$\tau$\end{tabular}}}}%
    \put(0,0){\includegraphics[width=\unitlength,page=2]{branches1.pdf}}%
  \end{picture}%
\endgroup
\caption{Example of a curve $\gamma^n_{(\s, -)}$ contained in a fundamental hand $\tau$ of level $n$, shown in red. The boundary of $\tau$ is represented with black continuous lines and is formed by ray tails. Dotted lines show other pieces of ray tails. Canonical tails that overlap with $\gamma^n_{(\s, -)}$ are displayed in different colours.}
\label{fig:branches}
\end{figure}

\item Suppose that $\gamma_{(\s, \ast)}^0 \subset \partial\tau \cap \partial \tilde{\tau}$, where $\tau$ and $\tilde{\tau}$ are two hands of level $n$. Let $x\in \gamma_{(\s, \ast)}^n$ be the point of smallest potential that also belongs to $\partial\tau \cap \partial \tilde{\tau}$. If $x$ is the endpoint of $\gamma_{(\s, \ast)}^n$, case \ref{item_two_hands} holds and we are done. Otherwise, by the same argument as before, $x$ must be a (preimage of a) critical point, and so, the continuation of $\gamma^n_{(\s, \ast)}$ towards points of smaller potential than the one of $x$, is a nested sequence of left or right bristles. By minimality of $x$, the bristle that contains $x$ can no longer be in the boundary of both $\tau$ and $\tilde{\tau}$, so it is either in the interior or in the boundary of only one of them. Then, we can argue as in the previous case and we see that \ref{item_unique_hand} holds.\qedhere
\end{itemize}
\end{proof}

Finally, we use the sets $\tau_n(\s, \ast)$ from the previous proposition to define, for any given signed address $(\s, \ast)$ and $n>0$, an inverse branch of $f$ in a (not necessarily open) neighbourhood $U$ of $\gamma^{n-1}_{(\sigma(\s), \ast)}$ such that $\tilde{f}(U)\supset \gamma^{n}_{(\s, \ast)}$.
\begin{thm}[Inverse branches for canonical tails]\label{lem_inverse_hand} Following Proposition~\ref{prop_hands_fortails}, for each $n\geq 0$ and $(\s, \ast)\in \Addr(f)_\pm$:
\begin{enumerate}[label=(\arabic*)]
\item \label{item1} There exists an open interval of signed addresses $\mathcal{I}^n(\s,\ast) \ni (\s,\ast)$ such that $$\tau_n(\ul{\alpha}, \star) \subseteq \tau_n(\s, \ast) \quad \text{ for all } \quad (\ul{\alpha}, \star) \in \mathcal{I}^n(\s,\ast).$$
\item \label{item2} If $n\geq 1$, the restriction $f\vert_{\tau_n(\s, \ast)}$ is injective and maps to $\tau_{n-1}(\sigma(\s), \ast)$.	
\end{enumerate}
Hence, for all $n\geq 1$, we can define the inverse branch
\begin{equation}\label{eq_inversebranch}
f^{-1,[n]}_{(\s,\ast)}\defeq \left(f\vert_{ \tau_n(\s, \ast)}\right)^{-1} \colon f(\tau_n(\s, \ast)) \longrightarrow \tau_n(\s, \ast).
\end{equation}
\end{thm}
\begin{proof}
We start by showing \ref{item2}. Recall that by Theorem \ref{thm_signed}, $f\vert_{\gamma^n_{(\s, \ast)}}$ maps injectively to $\gamma^{n-1}_{(\sigma(\s), \ast)}$, and by Proposition \ref{prop_basic_hand}, $f \vert_{\text{int}(\tau_n(\s, \ast))}$ maps injectively into a fundamental hand of level $n-1$, where $\text{int}(\tau_n(\s, \ast))$ denotes the interior of $\tau_n(\s, \ast)$. Suppose for the sake of contradiction that there exist $x\in \text{int}(\tau_n(\s, \ast)) \setminus \gamma^n_{(\s, \ast)}$ and $y\in \gamma^n_{(\s, \ast)}\setminus \text{int}(\tau_n(\s, \ast))$ so that $f(x)=f(y)$. 
In particular, $f(y)\in (X_n\cup \partial f(\tau_n(\s, \ast)))$, but this would contradict $x\in \text{int}(\tau_n(\s, \ast))$. Thus, $f\vert_{\tau_n(\s, \ast)}$ is injective. 

If $\gamma^0_{(\s, \ast)} \subset \text{int}(\tau_n(\s, \ast))$, then case \ref{item_unique_hand} in Proposition \ref{prop_hands_fortails} must occur for both $\gamma^n_{(\s, \ast)}$ and $\gamma^{n-1}_{(\sigma(\s), \ast)}$. Then, since $f$ is continuous, $f(\tau_n(\s, \ast)) \subset \tau_{n-1}(\sigma(\s), \ast)$ and \ref{item2} follows. Otherwise, $\gamma^0_{(\s, \ast)} \subset \partial \tau_n(\s, \ast) \cap \partial \tilde{\tau}$ for $\tilde{\tau}$ another hand of level $n$. If $\gamma^0_{(\s, \ast)}\subset f^{-1}(X_n)$, then $f(\gamma^0_{(\s, \ast)})$ is totally contained in a fundamental hand of level $n-1$, namely $\tau_{n-1}(\sigma(\s), \ast)$. Then, by continuity of $f$, both $\tau_n(\s, \ast)$ and $\tilde{\tau}$ are mapped under $f$ to $\tau_{n-1}(\sigma(\s), \ast)$, and \ref{item2} follows. Hence, we are left to study the case when $\gamma^0_{(\s, \ast)}\nsubseteq f^{-1}(X_n)$, which implies that $\gamma^{0}_{(\sigma(\s), \ast)}$ must also belong to $\partial \tau_{n-1}(\sigma(\s), \ast)$. By Proposition \ref{prop_hands_addr}, we can choose a pair of addresses $\ul{a}, \ul{b} \in \Addr(f)$ such that $\gamma^0_{\ul{a}}\subset \tau_n(\s, \ast)$ and $\gamma^0_{\ul{b}} \subset \tilde{\tau}$. In particular, by Proposition \ref{prop_hands_fortails}, it must occur that $\tau_n(\ul{a},\ast)=\tau_n(\s, \ast)$ and $\tau_n(\ul{b},\ast)=\tilde{\tau}$. We may assume without loss of generality that $[\ul{a}, \s, \ul{b}]_\ell$ holds, since the case when $[\ul{b}, \s, \ul{a}]_\ell$ does is analogous. Note that by \eqref{eq_cyclicorder}, this is equivalent to $[\gamma^0_{\ul{a}},\gamma^0_{\ul{s}},\gamma^0_{\ul{b}}]_{\infty}$, and since, by continuity of $f$, the circular order at infinity of these curves is preserved under iteration of $f$, by \eqref{eq_inclJs}, it holds that $[f(\gamma^0_{\ul{a}}),f(\gamma^0_{\ul{s}}),f(\gamma^0_{\ul{b}})]_{\infty}.$
Moreover, by definition of valid initial configurations, $f(\gamma^0_{\ul{\alpha}})\subset \gamma^0_{\sigma(\ul{\alpha})}$ for all $\ul{\alpha} \in \Addr(f)$. Hence, it holds $[\gamma^0_{\sigma(\ul{a})},\gamma^0_{\sigma(\ul{s})},\gamma^0_{\sigma(\ul{b})}]_{\infty}, $ and thus 
\begin{equation}\label{eq_preserveorder}
[\sigma(\ul{a}), \sigma(\s), \sigma(\ul{b})]_\ell.
\end{equation}
Recall that the sign $\ast \in \{-,+\}$ is preserved in the curve $\gamma^{n}_{(\s, \ast)}$ under the action of $f$, as $f(\gamma^{n}_{(\s, \ast)}) =\gamma^{n-1}_{(\sigma(\s), \ast)}$. Then, \eqref{eq_preserveorder} together with continuity of $f$ implies that if case \ref{item_two_hands} in Proposition \ref{prop_hands_fortails} holds for either $\gamma^{n}_{(\s, \ast)}$, $\gamma^{n-1}_{(\sigma(\s), \ast)}$, or both curves, then $\tau_n(\s, \ast)$ and $\tau_{n-1}(\sigma(\s), \ast)$ are chosen so that \ref{item2} holds. 

We prove \ref{item1} by induction on $n$. If $n=0$, then for any $(\s, \ast)\in \Addr(f)_\pm$, $\tau_0(\s, \ast)= F_0$ for some fundamental domain $F_0$. By Theorem \ref{thm_props_Js}, we can choose a pair of bounded addresses $\ul{a},\ul{b}$ whose first entry is $F_0$ and so that $[\ul{a}, \s, \ul{b}]_\ell$, and define $\mathcal{I}^0(\s, \ast)\defeq ((\ul{a},-), (\ul{b},+))$. Then, $\gamma_{(\ul{\alpha}, \star)}^0 \subset F_0$ for all addresses $(\ul{\alpha}, \star) \in \mathcal{I}^0(\s, \ast)$, and so $\tau_0(\ul{\alpha}, \star)=F_0$ and \ref{item1} follows.

Let us assume that the statement holds for all addresses in $\Addr(f)_\pm$ and some $n-1 \in \N$. Suppose that $\s=F_0F_1\ldots$ and note that by the inductive hypothesis, the interval $\mathcal{I}^{n-1}(\sigma(\s), \ast) \ni (\sigma(\s), \ast)$ is defined. If
\begin{equation}\label{eq_casogamma}
\gamma_{(\sigma(\s), \ast)}^{0} \subset \text{int}(\tau_{n-1}(\sigma(\s), \ast))\setminus X_n,
\end{equation} 
we can choose $\mathcal{I}'\subseteq \mathcal{I}^{n-1}(\sigma(\s), \ast)$ such that $\gamma_{(\ul{\alpha}, \star)}^{0} \subset \text{int}(\tau_{n-1}(\sigma(\s), \ast))\setminus X_n$ for all $(\ul{\alpha}, \star)\in \mathcal{I}'$ and $(\sigma(\s), \ast) \in \mathcal{I}'$. Then, for each $(\ul{\alpha}, \star)\in \mathcal{I}'$, by Proposition \ref{prop_hands_fortails}, there exists a unique fundamental hand $\tilde{\tau}$ of level $n$ such that $f(\tilde{\tau})\subset \text{int}(\tau_{n-1}(\sigma(\s), \ast))$ and $\gamma^0_{F_0\ul{\alpha}} \subset \tilde{\tau}$, and by continuity of $f$, the hand $\tilde{\tau}$ must equal $\tau_n(F_0\ul{\alpha}, \star)=\tau_n(\s, \ast)$. Hence, all the addresses in the set
$$S\defeq \{ (F_0\ul{\alpha}, \star) \in \Addr(f)_\pm : (\ul{\alpha}, \star) \in \mathcal{I}'\}$$ 
satisfy the property required. Moreover, $(\s, \ast)\in S$ and $S$ is an interval of addresses, since by continuity, $f$ preserves the order at infinity of extensions of level $0$. More specifically, if $\mathcal{I}'= ((\ul{a},\ast), (\ul{b},\star))$ for some $\underline{a}, \underline{b} \in \Addr(f)$, then $S=((F_0\ul{a},\ast), (F_0\ul{b},\star))\eqdef \mathcal{I}_n(\s, \ast).$

Otherwise, if \eqref{eq_casogamma} does not hold for $\gamma_{(\sigma(\s), \ast)}^{0}$, then either $\gamma_{(\sigma(\s), \ast)}^{0}\subset X_n$, or $\gamma_{(\sigma(\s), \ast)}^{0} \subset \partial \tau_{n-1}(\sigma(\s), \ast)$. In the second case, the interval of addresses $\mathcal{I}^{n-1}(\sigma(\s), \ast)$ must be of the form 

\begin{equation*}
\mathcal{I}^{n-1}(\sigma(\s), \ast)=
\renewcommand{\arraystretch}{1.5}\left \{\begin{array}{@{}l@{\quad}l@{}}
((\s, -), (\ul{a}, \star)) & \text{if } \ast=+ \\
((\ul{a}, \star), (\s, +)) & \text{if } \ast=-
\end{array}\right.\kern-\nulldelimiterspace
\end{equation*}
for some $(\ul{a}, \star)\in \Addr(f)_\pm$. To see this, we might assume without loss of generality that $\ast=+$. Then, since $\gamma_{(\sigma(\s), -)}^{0} \in \partial \tau_{n-1}(\sigma(\s), \ast)$, any open interval of addresses containing $(\sigma(\s), -)$ would also have to contain signed addresses whose corresponding Julia constituents lie in another fundamental hand. Thus, $\mathcal{I}^{n-1}(\sigma(\s), \ast)$ must be an open interval containing $(\s, +)$ but not $(\s, -)$, and so must be of the form claimed. Then, arguing as before, we can find a subinterval of addresses $\mathcal{I}'\subseteq \mathcal{I}^{n-1}(\sigma(\s), \ast)$ such that the curve $\gamma_{(\ul{\alpha}, \star)}^{0} \subset \text{int}(\tau_{n-1}(\sigma(\s), \ast))\setminus X_n$ for all $(\ul{\alpha}, \star)\in \mathcal{I}'\setminus \{(\s, +)\}$. By the analysis of the previous case, if we let $\mathcal{I}'\defeq ((\s,+), (\ul{b},+))$, then for all addresses $(\ul{\alpha}, \star) \in ((F_0\s, +), (F_0\ul{b}, +))$, it holds that $\gamma^n_{(\ul{\alpha}, \ast)} \subset \tau_n(\ul{\alpha}, \star) = \tau_n(\ul{s}, \ast)$. Then, for the statement to hold, we include the address $(\s, +)$ on the interval by defining $\mathcal{I}^n_{(\s, \ast)}\defeq ((\s, -), (F_0\ul{b}, +))$.

We are left to consider the case when $\gamma_{(\sigma(\s), \ast)}^{0}\subset X_n$. Note that by definition of fundamental hands, this implies that $\gamma_{(\s, \ast)}^{0}\subset \partial \tau_n(\s, \ast)$. For the purposes of defining the desired interval, we can regard $\gamma^{0}_{(\sigma(\s), \ast)}$ as if it belonged to the boundary of its fundamental hand in order to apply the same reasoning as before. That is, we choose a subinterval
$\mathcal{I}'\subset \mathcal{I}^{n-1}(\sigma(\s), \ast)$ of the form
\begin{equation*}
\mathcal{I}'\defeq
\renewcommand{\arraystretch}{1.5}\left \{\begin{array}{@{}l@{\quad}l@{}}
((\s, -), (\ul{a}, \star)) & \text{if } \ast=+ \\
((\ul{a}, \star), (\s, +)) & \text{if } \ast=-
\end{array}\right.\kern-\nulldelimiterspace
\end{equation*}
for some $\underline{a}\in \Addr(f)$, and proceed as in the previous case. 
\end{proof}

We can now make explicit and justify the idea from the beginning of the section of finding for each $n\geq 0$, inverse branches of $f^n$ defined on neighbourhoods of canonical tails satisfying certain properties:
\begin{observation}[Chains of inverse branches]\label{obs_chain_inverse} Following Theorem \ref{lem_inverse_hand}, for each $n\geq 0$ and $(\s, \ast)\in \Addr(f)_\pm$, we denote
$$f^{-n}_{(\s,\ast)}\defeq \left(f^n\vert_{ \tau_n(\s, \ast)}\right)^{-1} \colon f^n(\tau_n(\s, \ast)) \rightarrow \tau_n(\s, \ast).$$
Then, by Theorem \ref{lem_inverse_hand}\ref{item2}, the following chain of embeddings holds:
$$ \tau_n(\s, \ast)\xhookrightarrow{ \; \;\;\; f \;\;\; \; } \tau_{n-1}(\sigma(\s), \ast)\xhookrightarrow{\; \;\; \; f\;\; \; \; } \tau_{n-2}(\sigma(\s), \ast) \xhookrightarrow{ \;\;\; \; f\;\;\; \; } \cdots \xhookrightarrow{\; \;\; \; f \; \; \;\; } \tau_{0}(\sigma^{n}(\s), \ast). $$
This means that we can express the action of $f^{-n}_{(\s, \ast)}$ in $f^n(\tau_n(\s, \ast))$ as a composition of functions defined in \eqref{eq_inversebranch}. That is, 
$$ \tau_n(\s, \ast)\xleftarrow{ f^{-1,[n]}_{(\s, \ast)} } f(\tau_n(\s, \ast))\xleftarrow{ f^{-1, [n-1]}_{(\sigma(\s), \ast)} } f^2(\tau_n(\s, \ast)) \xleftarrow{ f^{-1, [n-2]}_{(\sigma^2(\s), \ast)} } \cdots \xleftarrow{ f^{-1, [1]}_{(\sigma^{n-1}(\s), \ast)} } f^n(\tau_n(\s, \ast)). $$
More precisely,
$$f^{-n}_{(\s,\ast)} \equiv \left(f^{-1, [n]}_{(\s, \ast)}\circ f^{-1, [n-1]}_{(\sigma(\s), \ast)} \circ \cdots \circ f^{-1, [1]}_{(\sigma^{n-1}(\s), \ast)}\right)\!\Big\vert_{f^n(\tau_n(\s, \ast))}.$$
Moreover, combining this with Proposition \ref{prop_hands_fortails} and Theorem \ref{lem_inverse_hand}, for all $(\ul{\alpha}, \star) \in \mathcal{I}^n(\s,\ast)$,
$$f^n(\tau_n(\ul{\alpha}, \star))\subseteq f^n(\tau_n(\s, \ast)) \quad \text{ and } \quad f^{-n}_{(\s,\ast)}\big\vert_{f^n(\tau_n(\ul{\alpha}, \star))}\equiv f^{-n}_{(\ul{\alpha},\star)}\big\vert_{f^n(\tau_n(\ul{\alpha}, \star))}.$$ 
\end{observation}

\begin{proof}[Proof of theorem \ref{thm_signed_intro}]It follows from Theorem \ref{thm_signed}, Observation \ref{rem_mother} and Theorem \ref{prop_hands_fortails}.
\end{proof}
\bibliographystyle{alpha}
\bibliography{../biblioComplex}
\end{document}